\documentclass[12pt]{amsart}%{article}
\usepackage{amsmath}
\usepackage{amsfonts}
\usepackage{amssymb}
\usepackage{amsthm}
\usepackage{graphicx}
\usepackage{color}
\usepackage{eucal}
\usepackage{tikz}

\newtheorem{theorem}{Theorem}[section]

\newtheorem{corollary}[theorem]{Corollary}

\newtheorem{definition}[theorem]{Definition}
\newtheorem{example}[theorem]{Example}

\newtheorem{proposition}[theorem]{Proposition}
\newtheorem{remark}[theorem]{Remark}

\makeatletter
\def\ball{{B}}
\def\tto{\rightrightarrows}
\def\bx{\bar x}
\def\by{\bar y}
\def\bz{\bar z}

\def\bv{\bar v}
\def\bw{\bar w}

\def\lip{\mathop{\rm lip}\nolimits}
\def\sur{\mathop{\rm sur}\nolimits}
\def\reg{\mathop{\rm reg}\nolimits}
\def\lip{\mathop{\rm lip}\nolimits}
\def\lopen{\mathop{\rm lopen}\nolimits}
\def\subreg{\mathop{\rm subreg}\nolimits}
\def\semireg{\mathop{\rm semireg}\nolimits}
\def\psopen{\mathop{\rm popen}\nolimits}
\def\calm{\mathop{\rm calm}\nolimits}
\def\recess{\mathop{\rm incalm}\nolimits}

\def\con{\mathop{\rm cone}\nolimits}
\def\intt{\mathop{\rm int}\nolimits}
\def\dom{\mathop{\rm dom}\nolimits}

\def\gph{\mathop{\rm gph}\nolimits}

\def\dist{\mathop{\rm dist}\nolimits}

\def\diam{\mathop{\rm diam}\nolimits}
\def\core{\mathop{\rm core}\nolimits}

\begin{document}

\subjclass{ 47J05  47J07  49J52 
49J53  58C15}

% Title, in lower case, with no explicit linebreaks (\\).  If the title
% is too long to be used as a running head, add a short version of the
% title in brackets, as in \title[shorttitle]{fulltitle}.

\title{On Almost Regular Mappings}

%%%%%%%%%%%%%%%%%%%%%%%%%%%%%%
% Author names and addresses 
%%%%%%%%%%%%%%%%%%%%%%%%%%%%%%

% Provide one separate \author{...} \address{...} \email{....} entry for each
% author, i.e., do not combine multiple authors.  Separate address lines by double
% slashes.  Do not attach footnotes to author  names. (For acknowledgements use
% the "\thanks" construct below.)

\author{Radek Cibulka} 
\address{NTIS - New Technologies for the Information Society and Department of Mathematics, Faculty of Applied Sciences, University of West Bohemia, Univerzitn\'i 8, 301 00 Plze\v{n}, Czech Republic}
\email{cibi@kma.zcu.cz}

%\author{Abderrahim Hantoute}  
%\address{University of Alicante, Alicante, Spain}
%\email{hantoute@ua.es}
 
%\author{Tom\'a\v{s} Roubal}
%\dedicatory{In honor of Jaroslav Kurzweil and to the memory of \v Stefan Schw\'abik.}
%\address{NTIS - New Technologies for the Information Society and Department of Mathematics, Faculty of Applied Sciences, University of West Bohemia, Univerzitn\'i 8, 301 00 Plze\v{n}, Czech Republic}
%\email{roubalt@kma.zcu.cz}

%%%%%%%%%%%%%%%%%%%%
% Acknowledgements
%%%%%%%%%%%%%%%%%%%

% Use \thanks for acknowledgements as footnotes to the title page.  
% (Note that footnotes inside \author or \title macros are not
% allowed.)
%
% In case of multiple author papers, phrase the acknowledgement to 
% say "The first author was supported by ...  The second author was
% supported by ..."

\thanks{The author was supported by the grant of GA\v CR no. 20-11164L.}

%%%%%%%%%%%%%
% Abstract 
%%%%%%%%%%%%%
%
% Abstracts should not contain macros (so that they can be processed independently
% of the paper.) Avoid displayed math and references in the abstract.

%\dedicatory{Dedicated to the memory of Asen L. Dontchev our colleague, mentor, and friend.}

\begin{abstract}
Ioffe's criterion and various reformulations of it  have become a~standard tool in proving theorems guaranteeing various regularity properties such as metric regularity, i.e., the openness with a linear rate around the reference point, of a~(set-valued) mapping. We derive an analogue of it guaranteeing the almost openness with a linear rate of mappings acting in incomplete spaces and having non-closed graphs, in general. The main tool used is an approximate version of Ekeland's variational principle for a function that is not necessarily lower semi-continuous and is defined on an abstract (possibly incomplete) space. Further, we focus on the stability of this property under additive single-valued and set-valued perturbations.  
\end{abstract}

\keywords{Ekeland variational principle, metric regularity, almost openness, Ioffe criterion, perturbation stability}

\maketitle

\section{Introduction} \label{scIntroduction}

Ekeland's variational principle (EVP)  plays a fundamental role in modern non-linear analysis and was generalized  by many authors in various ways.  Although our notation is fairly standard, in case of any difficulty, the reader is encouraged to consult the end of this introduction. 

\begin{theorem}\label{thmEkelandWeakMetric} Let $(X,d)$ be a complete metric space. Consider  a point  $x \in X$ and 
a~lower semi-continuous function $\varphi: X \to [0,\infty]$ such that $0 < \varphi(x) < \infty$. Then there exists a~point $u \in X$
such that  $\varphi(u) + d(u,x) \leq \varphi(x)$,  hence $d(u,x) \le \varphi(x)-\inf \varphi(X)$, and
\begin{equation} \label{eqEkelandWeak}
  \varphi(u') + d(u',u) \geq \varphi(u)  
    \quad \mbox{for every} \quad 
    u'\in X.
\end{equation}
\end{theorem}

This formulation of the EVP, which is enough for applications in regularity theory, will be often referred to as a \emph{weak form} of the principle while the full version of it means that the inequality in \eqref{eqEkelandWeak} is strict provided that $u' \neq u$. We focus on the setting of the so-called extended quasi-metric spaces \cite{CR2020} which allows to cover and unify the results by S.~Cobza\c{s} in quasi-metric spaces \cite{Cobzas}; a~directional version of the principle in Banach spaces by  M.~Durea, M.~Pan\c{t}iruc, and R.~Strugariu  \cite{DPS2}; and a statement by L.P. Hai and P.Q. Khanh in partial metric spaces \cite{HK2020}. We emphasize that we leave aside many important works not having a tight connection with our investigation. In the second section, we establish an approximate version  of the EVP without  the assumption that the domain space is complete and that the function under consideration is lower semi-continuous. We also show that under these two additional assumptions our statement implies the (weak form of the) principles mentioned above. As noted in \cite{BMS2015} quasimetric (instead of metric) structure of the topological space of arguments is one of absolutely mandatory requirements for an appropriate extension of the EVP for its possible application to the psychological model considered therein.

Metric regularity, linear openness, and Aubin/pseudo-Lipschitz property of the inverse of a~given set-valued mapping $F: X \rightrightarrows Y$ from a metric space $(X,d)$ into another metric space $(Y, \varrho)$ are three \emph{equivalent} properties playing a~fundamental role in modern variational analysis \cite{AubinFrankowskaBook, BorweinZhuBook, DontchevBook, IoffeBook, MordukhovichBook, PenotBook}. Let us consider the problem:
\begin{equation} \label{eqInclusion}
  \mbox{Given } y \in Y  \mbox{ find an } x \in X \mbox{ such that } F(x) \ni y.  
\end{equation}
 We focus on a weaker version of the second mentioned property called \emph{almost openness with a linear rate} guaranteeing that there is a $u \in X$ such the distance $\dist(y, F(u))$, often called a \emph{residual}, is arbitrarily small and the inverse $F^{-1}$ of $F$ verifies a kind of Lipschitz continuity, that is, \eqref{eqInclusion} admits an approximate solution and the distance between $u$ and the solution set can be estimated by a~multiple of the residual.         In the third section, we define an abstract concept of the almost openness with a linear rate relative to a fixed set,  introduce the remaining two notions corresponding to metric regularity and a certain continuity property of the inverse, and prove their equivalence. In this way, we cover not only the approximate version of full metric regularity but also of two important and distinct weaker concepts called metric semiregularity and subregularity. In the spirit of \cite{IoffeBook}, we call them simply \emph{almost (semi/sub)regularity} while the triad(s) of stronger properties we call \emph{(semi/sub)regularity}.  

Ioffe's  criteria provide transparent proofs of regularity statements from the literature, e.g., see \cite{FP1987, IoffeBook, CF15}. Theorem~\ref{thmEkelandWeakMetric} implies the following example of such a statement, which can  be traced back to \cite{FP1987}:   
  
\begin{proposition}\label{propFabianPreiss} Let $(X,d)$ be a complete metric space and $(Y,\varrho)$ be a metric space. Consider  $c > 0$ and $r > 0$, a point $\bar x\in X$, and a continuous mapping $g: X\to Y$ defined on the whole $X$. Assume that for  each $u \in \ball_X(\bar{x},r)$ and each $y \in \ball_Y(g(\bar{x}),cr)$, with $0 < \varrho(g(u),y) < cr$, there is a point $u'\in X$ %, better than $u$,  
satisfying 
\begin{eqnarray}\label{eqBetter}
 c\,d(u,u') < \varrho(g(u),y)-\varrho(g(u'),y).
\end{eqnarray} 
Then $\ball_Y(g(x),ct) \cap \ball_Y(g(\bar{x}),cr) \subset g\big(\ball_X(x,t)\big)$ for each $x \in \ball_X(\bar{x}, r)$ and each $t  \in \big(0,r - d(x, \bar{x})\big)$.   
\end{proposition}
\noindent In the fourth section, we show that sufficient conditions for almost regularity of a~set-valued mapping $F$ require none of the following commonly used assumptions playing the role of completeness/continuity from the above statement:
\begin{itemize}
 \item[(S$_1$)] the graph of $F$ is (locally) complete;    
 \item[(S$_2$)] the domain space $X$ is complete and the graph of $F$ is (locally) closed.
\end{itemize}
We do not employ any additional tools such as lower semi-continuous envelopes, slope-based conditions, completion of the spaces and the graph of the mapping under consideration, etc., see \cite{CR2023}.  This essentially simplifies the proofs.  Moreover, if either (S$_1$) or (S$_2$) holds true then we obtain full regularity by using \cite[Proposition 2.40]{IoffeBook}.

In the fifth section, we investigate the stability of almost regularity around the reference point with respect to additive single-valued and set-valued perturbations. We make the presentation as close as possible to \cite{CR2023}. Despite the similarity to the case of full (semi/sub)regularity \cite{IoffeBook, CR2023}, the results presented are new and require neither (S$_1$) nor (S$_2$).  In particular, we cover (and slightly generalize) a global statement from \cite{ABB2020}.   

In the sixth section, we gather several probably well-known results on almost openness of linear and affine mappings scattered in the literature. These (simpler) mappings can be used to approximate the non-linear ones. Proposition~\ref{propAlmostOpen} seems to be new and it may be of the independent interest.     

In order to keep the consideration as clear as possible we do not state many results in the full generality. In the last section, we comment on possible extensions of the statements presented. Our work seem to provide a theoretic background for several application fields. For example, the (inexact non-smooth) Newton-type methods request to find an ``approximate'' solution to the linear (sub)problem at each step; hence almost (sub/semi)regularity seem to be able to replace the (sub/semi)regularity in the corresponding convergence results, cf. \cite[Theorem 5.1]{CFK2019}. In control theory, the constrained (approximate) controllability of a non-linear dynamic system can be deduced from the same property of the appropriate (linear) approximation of it. Roughly speaking, this corresponds to the statements guaranteeing the perturbation stability of (almost) openness applied to the reachability operators of the dynamic systems, which to a control assign the final state of the corresponding system, e.g. see \cite[Theorem 11]{C2011} along with comments and references therein.
All this goes beyond the scope of the current work and is left for the future consideration.  

\medskip 
\paragraph{\bf Notation and terminology.} The notation $a:=b$ or $b =:a$ meas that $a$ is defined by $b$.  We use the convention that $ \inf \emptyset: = \infty$; as we work with non-negative quantities that $\sup \emptyset: = 0$; and that $0 \cdot \infty = 1$. If a set is
a singleton we identify it with its only element, that is, we write $a$ instead of $\{ a \}$. 
A subset $C$ of a~vector space $X$ is said to be \emph{symmetric} if $-C=C$; \emph{circled} if $\lambda C \subset C$ for each
      $\lambda\in [-1,1]$;  \emph{absorbing} if for every $x\in X$ there is a $\delta > 0$ such that $\lambda x\in C$ for each $\lambda\in(-\delta,\delta)$. A convex set is circled if and only if
      it is symmetric.  The \emph{core} of $C$ and the \emph{cone generated by $C$} are the sets $ \core C := \{x\in C: \,C-x \mbox{ is absorbing}\}$   and $\mbox{\rm cone} \, C := \bigcup_{t \geq 0} tC$. 
  In a metric space $(X,d)$, by $\ball_X[x,r]$ and $\ball_X(x,r)$ we denote, respectively, the closed and the open ball centered at an $x\in X$ with a radius $r \in [0, \infty]$; in particular, $\ball_X(x,0) = \emptyset$, $\ball_X[x,0] = \{x\}$, and   $\ball_X[x,\infty] = \ball_X(x,\infty) = X$;  the \emph{distance from a point} $x \in X$ to a set $C \subset X$ is $\mbox{\rm dist} \, (x,C):= \inf \{d(x,u): \ u \in C \}$; the \emph{diameter} of $C$ is $\mbox{diam} \, C := \sup\{d(u,x): \ u, x \in C\}$; 
  and $\overline{C}$ denotes the \emph{closure} of $C$.
The \emph{closed unit ball} and the \emph{unit sphere} in a
normed space $X$ are denoted by $\ball_X$ and $\mathbb{S}_X$, respectively. 
% given a set $C \subset X$, the convex hull of $C$, the closed convex hull of $C$, and the cone generated by $C$ are denoted by $\mbox{\rm co} \, C$, $\overline{\mbox{\rm co}} \, C$,  and $\mbox{\rm cone} \, C$. 
The \emph{domain} of an extended real-valued function $\varphi: X \to \mathbb{R} \cup\{\infty\}$ is the set $ \dom \varphi := \{x \in X: \ \varphi(x) < \infty\}$. 

By  ${f: X \to Y}$ we denote a single-valued mapping  $f$ acting from a non-empty set $X$ into another non-empty set $Y$;
while  ${F:X \rightrightarrows Y}$ denotes a  mapping from $X$ into $Y$ which may be set-valued. The set $\mbox{\rm dom} \,  F:=\{ x \in X:\; F(x)\neq\emptyset\}$
is the {\it domain} of $F$, the  \emph{graph} of $F$ is the set $\mbox{\rm gph} \,  F: = \{
(x,y)\in X\times Y: \ y\in F(x)\}$, and the \emph{inverse} of $F$ is the mapping $Y \ni y \longmapsto \{x \in X : \ y\in F(x)\}=: F^{-1}(y) \subset X$;
thus $F^{-1} :Y\rightrightarrows X$. Given a subset $M$ of $X$, the \emph{image} of $M$ under $F$ is the set  $F(M) := \bigcup_{x\in M} F(x)$.
Let $(X,d)$ be a metric space, let $(Y, \varrho)$ be a linear metric space, let $U \times V \subset X \times Y$ be non-empty, and let  $(\bx,\by)\in\gph F$. The mapping $F$ is said to have  \emph{Aubin property on $U\times V$ with a constant $\ell>0$} if 
\begin{eqnarray*}
	F(u)\cap V\subset F(u')+ \ball_{{Y}}[0,\ell \, d(u,u')]\quad\text{for each}\quad u, u'\in U;
\end{eqnarray*}
 to have \emph{Aubin property around $(\bx,\by)$ with a constant $\ell>0$} if there is a neighborhood $U \times V$ of $(\bar{x}, \bar{y})$ in $X \times Y$, equipped with the product (box) topology,  such that $F$ has Aubin property on  $U\times V$ with the constant $\ell$; and to be \emph{Hausdorff-Lipschitz on $U$ with a constant $\ell > 0$}, if $F$ has Aubin property on $U\times Y$ with the constant $\ell$.  The \emph{Lipschitz modulus} of $F$ around $(\bar{x}, \by)$, denoted by $\lip F(\bar{x}, \by)$, is the infimum of the set of all $\ell > 0$  such that $F$ has Aubin property around $(\bx,\by)$ with the constant $\ell$. In particular, for a mapping $f:X \to Y$, defined in a vicinity of $\bar{x}$, we have $\bar{y} = f(\bar{x})$ and we simply write     $\lip f(\bar{x})$; for each $\ell >  \lip f(\bar{x})$, if there is any, $f$ is \emph{Lipschitz continuous around} $\bar{x}$ with the constant $\ell$, i.e., there is an $r > 0$ such that   
\begin{equation} \label{eqLip}
 \varrho(f(u),f(x)) \leq \ell \, d(u,x)  \quad \mbox{for each} \quad  u,x \in \ball_X(\bar{x},r).
\end{equation}

By $\mathcal{BL}(X,Y)$ we denote the (vector) space of all bounded linear operators $A$ from a 
normed space $(X, \|\cdot\|)$ into another normed space $(Y, \|\cdot\|)$; and  
$X^* := \mathcal{BL}(X,\mathbb{R})$.

\section{Variational principles of Ekeland type} \label{scEVP}

Since the terminology concerning the extended quasi-metric spaces is not unified we prefer to postulate several properties of a non-negative function $\eta$ defined on $X \times X$, where $X$ is a non-empty set; instead of giving particular names to the pair $(X, \eta)$ when $\eta$ obeys (some of) these properties.

\begin{definition}\label{defQuasiMetricSpaceAxioms} Consider a non-empty set $X$ and a function $\eta: X\times X\to[0,\infty]$. We say that $\eta$ satisfies assumption 
\begin{enumerate}
	\item[$(\mathcal{A}_1)$] provided that $\eta(x,x)=0$ for each $x\in X$;
	\item[$(\mathcal{A}_2)$] provided that  $\eta(u,x)\leq \eta(u,z)+\eta(z,x)$ whenever $u$, $x$, $z\in X$;
	\item[$(\mathcal{A}_3)$] provided that $\eta(u,x)>0$ whenever $u$, $x\in X$ are distinct;
	\item[$(\mathcal{A}_4)$]  provided that  for each  sequence $(u_k)_{k \in \mathbb{N}}$ in $X$ such that for each $\varepsilon>0$ there is an index $k_0=k_0(\varepsilon)$ such that for each $j$, $k \in\mathbb{N}$, with $k_0\leq k<j$, we have $\eta(u_j,u_k)<\varepsilon$; there is a point $u\in X$ such that $\eta(u,u_k) \to 0$ as $k\to \infty$.	
\end{enumerate}
The \emph{conjugate of} $\eta$ is the function ${\eta}^*: X \times X \to [0,\infty]$ defined by ${\eta}^*(x,u) := \eta(u,x)$, $x$, $u \in X$.   
\end{definition}

We will employ the following easy and well-known fact, cf. \cite{Cobzas, HK2020}.
\begin{remark} \label{rmkLSC1} \rm Let $X$ be a non-empty set and let $\eta: X\times X\to [0,\infty]$ satisfy $(\mathcal{A}_2)$. Consider points  $u$, $x \in X$ and  a sequence $(u_k)_{k \in \mathbb{N}}$ in $X$  such that $\eta(u,u_k) \to 0$ as $k \to \infty$. Then  
$
  \liminf_{k \to \infty} \eta (u_k, x) = \liminf_{k \to \infty}\big( \eta (u_k,x) + \eta(u, u_k) \big) \geq \eta (u,x) 
$
and 
$
  \limsup_{k \to \infty} \eta (x,u_k) \leq \lim_{k \to \infty}\big( \eta (x,u) + \eta(u, u_k) \big) = \eta (x,u) 
$. 
If, in addition, $\eta = \eta^*$, then the previous estimates imply that  $\eta (x,u) \geq \limsup_{k \to \infty} \eta (x,u_k) \geq \liminf_{k \to \infty} \eta (x,u_k) = \liminf_{k \to \infty} \eta (u_k, x) \geq \eta (u,x)=\eta (x,u)$; hence $\eta (x,u_k) \to \eta (x,u)$  as $k \to \infty$.
  
\end{remark}
 
 Consider a non-empty set $X$ and a function $\eta: X\times X\to [0,\infty]$ satisfying $(\mathcal{A}_1)$ and $(\mathcal{A}_2)$. Given $x\in X$ and $r>0$, we define sets
$$
\ball^\eta_X(x,r):=\lbrace u\in X : \eta(x,u)<r\rbrace     
\quad \text{and}\quad
\ball^\eta_X[x,r]:=\lbrace u\in X : \eta(x,u)\leq r\rbrace.
$$
We define the topology  $\tau_{\eta}$  on $X$ starting from the family $\mathcal{V}_{\eta}(x)$ of neighborhoods of an  arbitrary point $x\in X$:
\begin{eqnarray*}
	V\in \mathcal{V}_\eta(x) \quad \Longleftrightarrow \quad  V \supset \ball_{X}^{\eta}(x,r)  \quad \text{for some} \quad r>0.
\end{eqnarray*}
The convergence of a sequence $(x_k)_{k \in \mathbb{N}}$ to an $x \in X$ in $(X, \tau_{\eta})$  is characterized as follows:  $\lim_{k\to \infty} x_k=x$   
if and only if  $\eta(x,x_k)\to 0$ as $ k \to\infty$.
Note that  $\ball^\eta_X(x,r)$ is $\tau_{\eta}$-open and $\ball^\eta_X[x,r]$ is  $\tau_{\eta^*} \, $-closed but not necessarily $\tau_{\eta}$-closed. By Remark~\ref{rmkLSC1}, for each $x\in X$  the mapping $X\ni u\longmapsto \eta(u,x)$ is $\tau_\eta$-lower semi-continuous while the mapping $X\ni u\longmapsto \eta(x,u)$ is $\tau_\eta$-upper semi-continuous. 

Of course, any (positive multiple of a) metric $d$ on $X$ satisfies $(\mathcal{A}_1)-(\mathcal{A}_3)$; while $(\mathcal{A}_4)$ holds true when $(X, d)$ is a complete space. In this case, we omit the superscript in the definition of the ``balls'' above, that is, we write $\ball_X(x,r)$ and $\ball_X[x,r]$, respectively.   Let us present two less trivial examples  \cite{CDPS, DPS2, HK2020}. 
\begin{example} \label{exTL} \rm
 Let $(X, \| \cdot\|)$ be a Banach space. Given a non-empty set $L\subset \mathbb{S}_{X}$,   the \emph{directional minimal time function with respect to $L$} is the function
\begin{equation} 
X\times X\ni (x,u)\longmapsto T_{L}(x,u):=\inf\left\{  t\geq 0 : u-x\in t L\right\}.  \label{min_t_M}
\end{equation}
 Clearly, for each $u$, $x\in
X$, if $T_{L}(x,u)<\infty$ (which is equivalent to $u-x\in
\con L$), then
$$
T_{-L}(u,x)=T_{L}(x,u)=\Vert u-x\Vert. 
$$
Hence ${T_{L}} = (T_{-L})^*$ and $\ball^{T_L}_X(x,r) = \ball_X(x,r)\cap (x + \con L)$ for each $x \in X$ and $r > 0$. Clearly,  $T_{L}$ satisfies $(\mathcal{A}_1)$ and $(\mathcal{A}_3)$. 
If $\con L$ is convex, then $T_{L}$ also satisfies  $(\mathcal{A}_2)$. If $L$ is closed (thus so is $\con L$), then the function $T_L$ is lower semi-continuous and satisfies $(\mathcal{A}_4)$. 
\end{example} 

\begin{example} \label{exPM} \rm
 Let $X$ be a non-empty set. Assume that $\zeta: X \times X \to [0,\infty)$ is such that for each  $u$, $x$, $z \in X$ we have
 \begin{equation} \label{eqZeta}
  \zeta(x,x) \leq \zeta(x,u) 
  \quad \mbox{and} \quad 
  \zeta(x,u) \leq \zeta(x,z) + \zeta(z,u) - \zeta(z,z). % \ \mbox{for each} \ u,x,z \in X.
 \end{equation}
 The above two conditions, called small self-distances and the triangle inequality, appear in the definition of a partial metric, e.g., see \cite[Definition 2.1]{HK2020}. If $\zeta$ satisfies the latter inequality in \eqref{eqZeta} then so does $\zeta^*$. In particular, both  $\zeta$ and  $\zeta^*$ satisfy $(\mathcal{A}_2)$. Assumption $(\mathcal{A}_4)$ holds for $\zeta$ provided that $X$ is \emph{$0$-complete} \cite[Definition 2.2]{HK2020}. 
   
  \noindent  Put $\eta(x,u) := \zeta(x,u) - \zeta(x,x)$, $u$, $x \in X$. Let $r > 0$ and $u$, $x$, $z \in X$ be arbitrary. Then  $\ball^{\eta}_X(x,r) = \{u \in X: \zeta (x,u) < \zeta (x,x) + r \}$; and 
\begin{eqnarray*}
 (0 \leq) \quad  \eta(x,u) & =  & \zeta(x,u) - \zeta(x,x) \leq \zeta(x,z) + \zeta(z,u) - \zeta(z,z) - \zeta(x,x) \\
 &  = & \eta(x,z) + \eta(z,u).
\end{eqnarray*}
Hence the non-negative function $\eta$ satisfies  $(\mathcal{A}_2)$; and, of course,  also $(\mathcal{A}_1)$. Given a sequence $(u_k)_{k \in \mathbb{N}}$ in $X$, we have $\lim_{k \to \infty} \eta (u,u_k) = 0$ if and only if $\lim_{k \to \infty} \zeta(u, u_k) = \zeta (u,u)$.

Define the topology  $\tau_{\zeta}$  on $X$ starting from the family $\mathcal{V}_{\zeta}(x)$ of neighborhoods of an  arbitrary point $x\in X$:
\begin{eqnarray*}
	V\in \mathcal{V}_\zeta(x) \ \Leftrightarrow \  V \supset \ball_X^\eta(x,r) = \{u \in X: \ \zeta(x,u) < r + \zeta(x,x)\} \ \text{for some} \ r>0.
\end{eqnarray*}
Then the convergence of a sequence $(u_k)_{k \in \mathbb{N}}$ to a $u$ in $(X, \tau_{\zeta})$  is characterized as follows:  $\lim_{k\to \infty} u_k=u$   
if and only if  $\lim_{k \to \infty} \zeta(u, u_k) = \zeta (u,u)$.  
\end{example} 

Let us present the main result in this section. 

\begin{theorem}\label{thmApproxSequentialEkeland}  Let $X$ be a non-empty set such that there is a function $\eta:~X\times X \to [0,\infty]$ satisfying  $(\mathcal{A}_2)$.  Consider  a point $x \in X$ and a function $\varphi: X \to [0,\infty]$ such that $0 < \varphi(x) < \infty$. Then there exist an $N \in \mathbb{N}\cup \{\infty\}$ and points  $(u_k)_{k = 1}^{N}$ in $X$ such that $u_1 = x$; that 
\begin{itemize}
  \item[$(i)$]   $\varphi(u_{j}) + \eta(u_{j},u_k) <\varphi(u_k)$ for each $j$, $k \in \mathbb{N}$ with $1 \leq k < j \leq N$; and 
  \item[$(ii)$] for  each $\varepsilon \in \big(0, \varphi(x) \big)$ there is an index  $n = n(\varepsilon) \leq N$ such that 
$$
 \varphi(u') + \eta(u',u_{k}) > \varphi(u_{k}) - \varepsilon 
 \quad \mbox{whenever} \quad u'\in X \mbox{ and }  k \in \mathbb{N} \mbox{ with } N \geq k \geq n.
$$
\end{itemize}
\end{theorem}

\begin{proof} 	We follow a standard pattern and construct inductively appropriate points  $u_1,u_2,\dots$ in $\dom \varphi$. Let $u_1:=x$. Assume that $u_k\in \dom \varphi$ is already defined for an index $k\in \mathbb{N}$.  Let 
	$$
	\alpha_k:=\inf\lbrace \varphi(u'): \ u'\in X \quad\text{and}\quad \varphi(u')+\eta(u', u_k) < \varphi(u_k) \rbrace
	\quad \quad \quad (\geq 0).
	$$
    If $\alpha_k = \infty$, that is, $\varphi(u')+\eta(u', u_k) \geq \varphi(u_k)$ for each $u' \in X$; then stop the construction.  
    Otherwise,  find a point $u_{k+1}\in X$ such that
	\begin{eqnarray}
	\label{ApproxEkelEstimate1}
	\varphi(u_{k+1})+\eta(u_{k+1}, u_{k}) < \varphi(u_k)\quad\text{and}\quad \varphi(u_{k+1})<\alpha_k+ 1/k.
	\end{eqnarray}
	Note that  $\alpha_k \leq \varphi(u_{k+1})\leq \varphi(u_{k+1})+  \eta(u_{k+1},u_k) < \varphi(u_k)< \infty$. Hence $u_{k+1} \in \dom \varphi$  and we repeat the construction. 
	
	If the set of all iterates $u_k$ generated by the above process is finite, then let $N$ be the cardinality of this set; otherwise put $N := \infty$. 
	 By  $(\mathcal{A}_2)$, for all $j$, $k \in \mathbb{N}$, with $1\leq k <j \leq N$, we have that 
	\begin{eqnarray}
	0 \leq \eta(u_j,u_k) &\leq & \eta(u_{j},u_{j-1})+\dots + \eta(u_{k+1},u_{k})\nonumber\\
	&<& \big( \varphi(u_{j-1})-\varphi(u_{j})\big)+\cdots + \big( \varphi(u_{k})-\varphi(u_{k+1})\big) = \varphi(u_k)-\varphi(u_j).	\label{ApproxEkelEstimate2}
	\end{eqnarray}
	This proves (i).
	
	Let $\varepsilon \in \big(0, \varphi(x) \big)$ be arbitrary. If $N < \infty$, then put $n := N$ and we are done, because the iterative process stopped after a finite number of steps, that is, $\varphi(u')+\eta(u', u_N) \geq \varphi(u_N)$ for each $u' \in X$. Further assume that $N = \infty$.    The first inequality in \eqref{ApproxEkelEstimate1} implies that a (bounded from below) sequence $\big( \varphi(u_k)\big)$ is strictly decreasing, so it converges. Find an index $n=n(\varepsilon) > 2/\varepsilon$ such that $ 0 <  \varphi(u_{n})-\varphi(u_{n+1})  < \varepsilon/2$.
	%\begin{equation} \label{ApproxEkelEstimate3}
    % 0 \leq  \eta(u_{j},u_{k}) <  \varphi(u_{k})-\varphi(u_{j}) < \varepsilon/2 \quad \mbox{for each} \quad j  > k \geq n.
	 % .
	 %\end{equation}	 
    Assume, on the contrary, that there are $k \geq n$ and $u'\in X$ such that 
	$$
	  \varphi(u')+ \eta(u',u_{k})\leq \varphi(u_{k}) - \varepsilon.
	$$
    If $k = n$, then $\varphi(u')+ \eta(u',u_{n}) <  \varphi(u_{n})$ trivially. If $k > n$, then $(\mathcal{A}_2)$, the last displayed estimate, and \eqref{ApproxEkelEstimate2}, with $j:= k$ and $k := n$, imply that 
    \begin{eqnarray*}
      \varphi(u')+ \eta(u',u_{n}) & \leq &  \varphi(u') + \eta(u',u_{k}) + \eta(u_{k}, u_{n})  \\
      & < &  \varphi(u_{k}) - \varepsilon + \varphi(u_{n})-\varphi(u_{k})  =  \varphi(u_{n}) - \varepsilon \quad \quad \big(<\varphi(u_{n})\big).
    \end{eqnarray*}
     Hence   $\alpha_{n} \leq \varphi(u')$. The latter inequality in \eqref{ApproxEkelEstimate1} and the choice of $n$ yield that 
	\begin{eqnarray*}
	\varphi(u') & \leq &  \varphi(u')+ \eta (u',u_{n}) <   \varphi(u_{n}) - \varepsilon < \varphi(u_{n+1}) - \varepsilon/2 < \varphi(u_{n+1})-1/n \\
	&  <  &  \alpha_{n} \leq \varphi(u'),
	\end{eqnarray*}
	a contradiction.
\end{proof}

As an easy corollary we get:
\begin{corollary}\label{corApproxEkeland}  Let $X$ be a non-empty set such that there is a function $\eta:~X\times X \to [0,\infty]$ satisfying $(\mathcal{A}_2)$.  Consider an $\varepsilon > 0$,  a point $x \in X$, and  a function $\varphi: X \to [0,\infty]$ such that $\eta(x,x) < \infty$ and $\varepsilon < \varphi(x) < \infty$. Then there exists a~point $u  \in X$
such that  $\varphi(u) + \eta(u,x) \leq \varphi(x) + \eta(x,x)$ and
\begin{equation} \label{eqAppoxEkeland02}
 \varphi(u') + \eta(u',u) > \varphi(u) - \varepsilon 
 \quad \mbox{for every} \quad 
  u'\in X.
\end{equation}
\end{corollary}

\begin{proof}
  Use Theorem~\ref{thmApproxSequentialEkeland} to find an $N \in \mathbb{N} \cup \{\infty\}$, points  $(u_k)_{k =1}^{N}$ with  $u_1 = x$, and an index $n = n(\varepsilon)$ such that the conclusions of both (i) and (ii)  hold true. If $N = 1$ then $u:=u_1=x$ does the job. If $N > 1$, then put $u := u_{m}$, where $m := \max\{2,n\}$. Then (ii), with $k := m$, implies \eqref{eqAppoxEkeland02}. By (i),  with $k := 1$ and $j:=m$, we infer that   $\varphi(u) + \eta(u,x) < \varphi(x) \leq \varphi(x) + \eta(x,x)$.
\end{proof}

Next statement slightly improves a weak form of the EVP  from \cite{CR2020} covering results in \cite{Cobzas, DPS2}.

\begin{theorem}\label{thmEkelandWeak} Let $X$ be a non-empty set such that there is a function $\eta:~X\times X\to [0,\infty]$ satisfying  $(\mathcal{A}_2)$ and $(\mathcal{A}_4)$.  Consider a point $x \in X$ and 
 a function $\varphi: X \to [0,\infty]$ such that $\eta(x,x) < \infty$ and $0 < \varphi(x) < \infty$. Assume that for each $u \in X$ and each $(u_k)_{k \in \mathbb{N}}$  in $X$ such that  $(\varphi(u_{k}))_{k \in \mathbb{N}}$ is strictly decreasing and $\eta(u,u_k) \to 0$ as $k \to \infty$ we have that  
 \begin{equation} \label{eqSeqStricLSC}
    \varphi(u) \leq \lim_{k \to \infty} \varphi(u_k) 
    \quad \quad \quad \quad  (\mbox{yes, the limit exists}).
 \end{equation}
  Then there exists a~point $u \in X$
such that  $\varphi(u) + \eta(u,x) \leq \varphi(x) + \eta(x,x)$ and
\begin{equation} \label{eqEVPWeak02}
 \varphi(u') + \eta(u',u) \geq \varphi(u) 
 \quad \mbox{for every} \quad 
  u'\in X.
\end{equation}
\end{theorem}
 
\begin{proof}
   Using Theorem~\ref{thmApproxSequentialEkeland}, we find an $N \in \mathbb{N} \cup \{\infty\}$ and points $(u_k)_{k =1}^{N}$ such that $u_1 = x$ and (i) holds true. 
  
  First, assume that $N = \infty$.  In view of (i), the sequence $\big(\varphi(u_k)\big)$ is strictly decreasing (and bounded from below); hence it converges.  As $\big(\varphi(u_k)\big)$  is a Cauchy sequence, combining (i) and $(\mathcal{A}_4)$, we find a $u \in X$ such that $\lim_{k \to \infty} \eta(u,u_k) =0$. As  $u_1 = x$, employing \eqref{eqSeqStricLSC}, Remark~\ref{rmkLSC1},   and (i) with $k:=1$, we get 
  \begin{eqnarray*} %\label{eqEVP03}
   \varphi(u) + \eta(u,x) & \leq & \lim_{j \to \infty} \varphi(u_{j}) + \liminf_{j \to \infty}  \eta(u_{j},x) \notag \\
   & = & \liminf_{j \to \infty} \big(\varphi(u_{j}) + \eta(u_{j},x)\big) 
    \leq  \varphi(x)\leq \varphi(x) + \eta(x,x).
  \end{eqnarray*}
     To prove the latter inequality, let $u' \in X$ be arbitrary. Pick any $\varepsilon > 0$. Find an index $n = n(\varepsilon)$ such that the conclusion in (ii) holds true.  Then
   \begin{eqnarray*} 
    \varphi(u') + \eta(u',u)  & = &  \lim_{k \to \infty} \big(\varphi(u') + \eta(u',u) + \eta(u,u_{k})  \big)  \\
    &   \overset{(\mathcal{A}_2)}{\geq} &  \liminf_{k \to \infty} \big(\varphi(u') + \eta(u',u_{k}) \big)   \overset{(ii)}{\geq}  \lim_{k \to \infty}  \varphi(u_{k}) - \varepsilon 
     \overset{\eqref{eqSeqStricLSC}}{\geq}   \varphi(u) - \varepsilon. 
  \end{eqnarray*}
  As $\varepsilon > 0$ is arbitrary, letting  $\varepsilon \downarrow 0$, we get that $\varphi(u') + \eta(u',u) \geq  \varphi(u)$. 
  
  Second, assume that $ N < \infty$. Put $u := u_{N}$. If $N > 1$ then (i), with $k:=1$ and $j := N$, implies that $\varphi(u) + \eta(u,x) < \varphi(x) \leq \varphi(x) + \eta(x,x)$. If $N = 1$ then the first desired inequality holds trivially because  $u=u_1 = x$. By (ii),  we have $\varphi(u') + \eta(u',u) > \varphi(u) - \varepsilon$ for each $\varepsilon$ small enough, which yields the second inequality as in case $N = \infty$. 
\end{proof}

\begin{remark} \label{rmkLSC2} \rm Of course, Theorem~\ref{thmEkelandWeakMetric} follows from  the last result. More generally, let $X$ be a non-empty set and let $\eta: X\times X\to [0,\infty]$ satisfy  $(\mathcal{A}_1)$ and $(\mathcal{A}_2)$. Assume that $\varphi: X \to [0, \infty]$  a~$\tau_\eta$-lower semi-continuous function. Then the assumption containing \eqref{eqSeqStricLSC}, called strict-decreasing-lower semi-continuity in \cite{HK2020}, holds true; see \cite[Example 2.9]{HK2020} illustrating that this notion is weaker than the usual lower semi-continuity in the setting from Example~\ref{exPM}.     
\end{remark}

Note that the last theorem  is equivalent to the most popular form of the EVP involving two constants and three conclusions (and the preceding results can be modified accordingly).  

\begin{theorem}\label{thmEkelandWeakE}
Let $X$ be a non-empty set such that there is a function $\eta:~X\times X\to [0,\infty]$ satisfying  $(\mathcal{A}_2)$ and $(\mathcal{A}_4)$.  Consider $\delta >0$ and $r>0$, a~point $x \in X$, and 
 a function $\varphi: X \to (\infty,\infty]$ such that  $\eta(x,x) < \infty$ and $-\infty <  \inf \varphi(X)  < \varphi(x) \leq \inf \varphi(X)  +  \delta$. Assume that for each $u \in X$ and each $(u_k)_{k \in \mathbb{N}}$  in $X$ such that  $(\varphi(u_{k}))_{k \in \mathbb{N}}$ is strictly decreasing and $\eta(u,u_k) \to 0$ as $k \to \infty$ inequality \eqref{eqSeqStricLSC}  holds true.   
Then there is a point $u \in X$ such that
\begin{enumerate}
 \item[$(a)$] $\varphi(u) + \hbox{$\frac{\delta}{r}$} \, \eta(u,x) \leq \varphi(x) + \hbox{$\frac{\delta}{r}$} \, \eta(x,x)$;
 \item[$(b)$] $\varphi(u') +\hbox{$\frac{\delta}{r}$} \, \eta(u',u) \geq \varphi(u)$ for every  $u'\in X$; and
 \item[$(c)$] $\eta(u,x) \leq r + \eta (x,x)$.
 \end{enumerate}
\end{theorem}

\begin{proof} 
Applying Theorem~\ref{thmEkelandWeak}, with $\varphi := \varphi - \inf{\varphi(X)}$ and $\eta := \hbox{$\frac{\delta}{r}$}  \eta$, and adding  $\inf{\varphi(X)}$, we find a point $u \in X$ such that  $ \varphi(u) + \hbox{$\frac{\delta}{r}$} \, \eta(u,x) \leq \varphi(x) + \hbox{$\frac{\delta}{r}$} \, \eta(x,x)$ and
$ \varphi(u') +\hbox{$\frac{\delta}{r}$} \, \eta(u',u) \geq \varphi(u)$ for every  $ u'\in X$.
Hence both (a) and (b) hold true. By (a), we get that $\delta \, \big(\eta(u,{x}) - \eta(x,x)\big) \leq r \big(\varphi(x) - \varphi(u)\big) \leq r \big(\varphi({x}) - \inf{\varphi(X)} \big)  \leq r\delta $, that is, $\eta(u,x) \leq r + \eta (x,x)$.
\end{proof}

To conclude this section, let us comment on several existing results in the literature in connection with Theorems \ref{thmEkelandWeak} and \ref{thmEkelandWeakE}. 

\begin{remark} \label{rmkComentarySC2} \rm
 Choosing $\eta: X\times X\to [0,\infty]$ verifying $(\mathcal{A}_1)-(\mathcal{A}_4)$, we get weak forms of Theorem 2.4 and Corollary 2.8 in \cite{CR2020}; hence, using the latter statement,  of \cite[Theorem 2.1]{BKP2005} as well. In particular, we get weak forms of \cite[Corollary 3.2]{DPS2}, where  $\eta := T_L$ from Example~\ref{exTL}; and of \cite[Theorem 2.4]{Cobzas} if $\eta$ is a (finitely-valued) quasi-semimetric (instead of a quasi-metric generating T$_1$-topology used therein). 
\end{remark}
 
 \begin{remark} \label{rmkComentarySC3} \rm
 Let $\zeta$ and  $\eta$ be as in Example~\ref{exPM}. Then  $\zeta$ satisfies  $(\mathcal{A}_2)$ and if, in addition, $\zeta = \zeta^*$ and $X$ is $0$-complete, then  Theorem~\ref{thmEkelandWeakE}, with $\eta:= \zeta$,  yields  \cite[Theorem 4.1]{HK2020}. Indeed, since $\zeta(u,x)=\zeta(x,u)$, from (a)  and \eqref{eqZeta} we get that $\varphi(u) \leq   \varphi(x) + \hbox{$\frac{\delta}{r}$} (\zeta(x,x) - \zeta(u,x))  \leq \varphi(x)$ which is (ii) therein.  Note that we do not need all the properties of the partial metric from \cite[Definition 2.1]{HK2020}.
 
 Further, $\eta$  verifies $(\mathcal{A}_1)$ and $(\mathcal{A}_2)$ by \eqref{eqZeta}. If $(\mathcal{A}_4)$ holds then  Theorem~\ref{thmEkelandWeakE} yields  a point $u \in X$ such that $0 \leq\zeta(u,x) - \zeta(u,u) \leq r$; that $\varphi(u) \leq   \varphi(u) + \hbox{$\frac{\delta}{r}$} (\zeta(u,x) - \zeta(u,u)) =  \varphi(u) + \hbox{$\frac{\delta}{r}$} \, \eta(u,x)  \leq \varphi(x)$; and that  
$ \varphi(u') +\hbox{$\frac{\delta}{r}$} \, \zeta(u',u) \geq  \varphi(u') +\hbox{$\frac{\delta}{r}$} \, \eta(u',u) \geq \varphi(u)$ for every  $u'\in X$. We obtain a stronger conclusion than  \cite[Theorem 4.1]{HK2020}; but the assumption $(\mathcal{A}_4)$ seems to be stronger than $0$-completeness. 
\end{remark}

\section{Almost regularity, subregularity, and semiregularity} \label{scRegularityProperties}

We start with approximate versions of the abstract properties from \cite[Definitions 2.21, 2.22, and 2.23]{IoffeBook}, which are equivalent to each other in the sense of Proposition~\ref{propApproxRegularitiesEquiv} below; and such that each property covers three substantially different concepts defined later. 
\begin{definition} \label{defApproxRegularities}
Let  $(X,d)$ and $(Y, \varrho)$ be metric spaces. Consider a $\mu > 0$,  non-empty sets $U \subset X$ and $V \subset Y$, a function $\gamma: X \to[0, \infty]$ defined and not identically zero on $U$, and a set-valued mapping  $G:X \rightrightarrows Y$. The mapping $G$ is said to have for $\mu$ and $\gamma$ on $U \times V$  property
\begin{itemize}
  \item[$(\mathcal{O})$]  provided that $
\ball_Y(y,\mu t) \cap V  \subset \overline{G\big(\ball_X(x,t)\big)}$ for each $ x \in U$, each $y \in G(x)$,  and each $ t  \in (0,\gamma(x))$;
 \item[$(\mathcal{R})$]  provided that $ \lim_{\varepsilon\downarrow 0} 	\dist\left(x,G^{-1}\big(\ball_Y(y,\varepsilon)\big)\right)\leq \mu  \dist (y, G(x))$ for each  $(x,y)  \in U \times  V$ such that  $\mu  \dist (y, G(x)) < \gamma(x)$;
   \item[$(\mathcal{L}^{-1})$]   provided that $ \lim_{\varepsilon\downarrow 0} 	\dist\left(x,G^{-1}\big(\ball_Y(y',\varepsilon)\big)\right)\leq \mu\, \varrho(y,y')$ for each  $y\in Y$, each $y'  \in V$,  and each $x \in G^{-1}(y)\cap U$ such that  $\mu\, \varrho(y,y') < \gamma(x)$. 
\end{itemize}
\end{definition}

Now we present an analogue of \cite[Theorem 2.25]{IoffeBook}.   

\begin{proposition} \label{propApproxRegularitiesEquiv}  Let  $(X,d)$ and $(Y, \varrho)$ be metric spaces. Consider a $c > 0$, non-empty sets $U \subset X$ and $V \subset Y$, a function $\gamma: X \to(0, \infty]$ defined and not identically zero on $U$, and  a set-valued mapping  $G:X \rightrightarrows Y$. Then  $G$ has property $(\mathcal{O})$ on $U \times V$ for $c$ and $\gamma$ if and only if $G$ has property $(\mathcal{R})$ on $U \times V$ for $1/c$ and $\gamma$ if and only if  $G$ has property  $(\mathcal{L}^{-1})$ on $U \times V$ for $1/c$ and $\gamma$.
\end{proposition}

\begin{proof} Put $\mu := 1/c$. 
Assume that $G$ has property $(\mathcal{O})$ on $U \times V$ for $c$ and $\gamma$. Fix arbitrary $x \in U$ and $y \in V$ with
$\mu \dist (y, G(x)) < \gamma(x)$. Pick an arbitrary $\varepsilon > 0$ such that  $t:= \mu (\dist (y, G(x)) + \varepsilon) \in (0, \gamma(x))$.   Find a $v \in G(x)$ such that $ \varrho(y,v) < \dist (y, G(x)) + \varepsilon$. Then $y \in \ball_Y(v,ct) \cap V$. By property $(\mathcal{O})$, there are $w \in Y$ and $u\in \ball_X(x,t)$ such that $\varrho(y,w) < \varepsilon$ and $w \in G(u)$. As $ u\in G^{-1}(w) \subset G^{-1}\big(\ball_Y(y,\varepsilon)\big)$, we have  
\begin{eqnarray*}
\dist\left(x,G^{-1}\big(\ball_Y(y,\varepsilon)\big)\right) \leq d(x,u)  < t = \mu (\dist (y, G(x)) + \varepsilon) \quad \quad \big(< \gamma(x)\big).
\end{eqnarray*}
The quantity on the left-hand side of the inequality above is increasing and bounded from above by $\gamma(x)$ when $\varepsilon \downarrow 0$. Letting $\varepsilon \downarrow 0$, and remembering that $x \in U$ and $y \in V$ with
$\mu \dist (y, G(x)) < \gamma(x)$ are arbitrary,  we conclude that $G$ has property $(\mathcal{R})$ on $U \times V$ for $\mu$ and $\gamma$.  

Assume that $G$ has property $(\mathcal{R})$ on $U \times V$ for $\mu$ and $\gamma$.  Fix arbitrary $y \in Y$, $y' \in V$, and  $x \in G^{-1}(y)\cap U$ with  $\mu \,  \varrho (y,y') < \gamma(x)$.  Then $\mu   \dist(y', G(x)) \leq \mu \,  \varrho (y',y) < \gamma(x)$. So $ \lim_{\varepsilon\downarrow 0} 	\dist\left(x,G^{-1}\big(\ball_Y(y',\varepsilon)\big)\right)\leq \mu   \dist (y', G(x)) \leq  \mu \varrho (y',y) $. Hence $G$ has property $(\mathcal{L}^{-1})$ on $U \times V$ for $\mu$ and $\gamma$.

Assume that $G$ has property $(\mathcal{L}^{-1})$ on $U \times V$ for $\mu$ and $\gamma$. Let $x\in U$, $y \in G(x)$, and $t\in (0,\gamma(x))$ be arbitrary. Pick any  $v\in\ball_Y\left(y, c t\right)\cap V$, if there is such. Then $  \mu \varrho(v,y) =  \varrho(v,y)/c < t < \gamma(x)$. Using property $(\mathcal{L}^{-1})$, for each $\varepsilon > 0$, small enough, we have   
$
 \dist\left(x,G^{-1}\big(\ball_Y(v,\varepsilon)\big)\right) < t;
$
so there is a $u_{\varepsilon}   \in X$  such that $G(u_{\varepsilon} ) \cap \ball_Y(v,\varepsilon) \neq \emptyset$  and  $d(x,u_{\varepsilon} ) < t$; and therefore
$$
  \dist\big(v,G( \ball_X(x,t))\big) \leq \dist(v,G(u_{\varepsilon} ))  < \varepsilon.
$$
Consequently,  $v \in \overline{G\big(\ball_X(x,t)\big)}$. Since $v\in\ball_Y\left(y, c t\right)\cap V$ as well as $x\in U$, $y \in G(x)$, and $t\in 
  (0,\gamma(x))$ are arbitrary, we conclude that  $G$ has property $(\mathcal{O})$ on $U \times V$ for $c$ and $\gamma$. The loop is closed.
\end{proof}

It is an old, well-known, although nowadays almost not used, fact that any regularity property of a set-valued mapping can be deduced from the same property of a particular single-valued mapping; namely, the restriction of the canonical projection from $X \times Y$ into $Y$ to the graph of the set-valued mapping, see e.g. \cite[Proposition 2.33]{IoffeBook}. The same is true for  our approximate version:    
\begin{remark} \label{rmkProjection} \rm Let all the assumptions of Proposition~\ref{propApproxRegularitiesEquiv} be satisfied and let $\alpha \in (0, 1/c)$ be arbitrary. Define a (compatible) metric $\omega$ on $X \times Y$ by 
$$
  \omega\big((u,v), (x,y) \big) := \max\big\{d(u,x), \alpha \varrho(v,y) \big\}, 
  \quad (u,v), (x,y) \in X \times Y.  
$$
 {\it Then $G$ has property $(\mathcal{O})$  on $U \times V$  for $c$ and $\gamma$  if and only if a mapping $g : X \times Y \to Y$ defined by $g(x,y) = y$, $(x,y) \in \gph G$, has property $(\mathcal{O})$  on $\big((U \times Y)\cap \gph G \big)\times V$  for $c$ and $\widetilde{\gamma}(x,y) := \gamma(x)$, $(x,y) \in X \times Y$.} \\
Assume that $G$ has property $(\mathcal{O})$  on $U \times V$  for $c$ and $\gamma$. Fix arbitrary $(x,y) \in (U \times Y) \cap \gph G$ and $t \in (0, \widetilde{\gamma}(x,y))$. So $x \in U$, $g(x,y)=y \in G(x)$, and $t < \gamma(x)$. Fix any $v \in   \ball_Y(y, c t) \cap V$. Then $v \in  \overline{G\big(\ball_X(x,t)\big)}$. Let $\varepsilon \in (0, t(1-\alpha c)/\alpha)$ be arbitrary. There are $w \in Y$ and  $u \in \ball_X(x,t)$ such that  $ \varrho(v, w) < \varepsilon$  and  $w \in G(u)$. Then $\varrho(v, g(u,w)) =  \varrho(v, w) < \varepsilon$. As $d(x,u) < t$ and $\varrho(y,w) \leq \varrho(y,v) + \varrho(v,w) <  c t + \varepsilon < t/\alpha$, we have
$\omega\big((x,y), (u,w) \big) < t$. Since $\varepsilon$ can be arbitrarily small, we get $v \in \overline{g\big( \ball_{X\times Y}^{ \, \omega}((x,y),t) \big)}$; for any $v \in   \ball_Y(y, c t) \cap V$. \\
 On the other hand, assume that $g$ has property $(\mathcal{O})$  on $\big((U \times Y)\cap \gph G \big)\times V$  for $c$ and $\widetilde{\gamma}$. Fix arbitrary $x \in U$,  $y \in G(x)$, and $t \in (0, \gamma(x))$.   Pick any $w \in   \ball_Y(y, c t) \cap V = \ball_Y(g(x,y), c t) \cap V $. Given an $\varepsilon > 0$, there is a $(u,v) \in X \times Y$, with $v \in G(u)$, such that $d(u,x) < t$, $\alpha \varrho (v,y) < t$, and $ \varepsilon > \varrho(w, g(u,v)) = \varrho (w,v)$; thus $\dist \big(w,G(\ball_X(x,t))\big) \leq \dist(w, G(u)) \leq \varrho(w,v) < \varepsilon$. So $w \in \overline{G(\ball_X(x,t))}$.  
\end{remark}

Letting $U$ and $V$ be the neighborhoods of a fixed reference point in the graph of $G$, we  get the notion of  \emph{almost regularity around} this point:  
 
\begin{definition} \label{defApproxRegularity}
Let  $(X,d)$ and $(Y, \varrho)$ be metric spaces. Consider a mapping  $G:X \rightrightarrows Y$ and a point $(\bar{x}, \bar{y}) \in \gph G$. Then $G$ is said to
\begin{itemize} 
   \item[$(i)$]  be  \emph{almost open with a constant $c > 0$ around} $(\bar{x}, \bar{y})$
 if there is a $\gamma > 0$ such that  %$G$ has property $(\mathcal{O})$ on $\ball_X(\bar{x}, \gamma) \times \ball_Y(\bar{y}, \gamma)$ for $c$ and $\gamma(\cdot) \equiv \gamma$, that is, 
\begin{equation}\label{defApproxOpen}
\ball_Y(y,ct) \cap \ball_Y(\bar{y}, \gamma)  \subset \overline{G\big(\ball_X(x,t)\big)}
\end{equation}
 for each $ x \in \ball_X(\bar{x}, \gamma)$, each $y \in G(x)$,  and each  $t  \in (0,\gamma)$;
 \item[$(ii)$]  be  \emph{almost metrically regular with a constant  $\kappa > 0$ around} $(\bar{x}, \bar{y})$ if there is a $\gamma > 0$ such that 
 % $G$ has property $(\mathcal{R})$ on $\ball_X(\bar{x}, \gamma) \times \ball_Y(\bar{y}, \gamma)$ for  $\kappa$ and $\gamma(\cdot) \equiv \gamma$, that is,  
\begin{equation} \label{defApproxMetricRegularity}
\lim\nolimits_{\varepsilon\downarrow 0} 	\dist\left(x,G^{-1}\big(\ball_Y(y,\varepsilon)\big)\right)\leq \kappa \, \dist(y, G(x))
\end{equation}
for each $x \in\ball_X(\bx,\gamma)$ and each $y \in \ball_Y(\bar{y},\gamma)$ with  $\kappa  \dist(y, G(x)) < \gamma$;
   \item[$(iii)$] have an \emph{almost pseudo-Lipschitz inverse with a constant $\ell > 0$ around} $(\bar{x}, \bar{y})$
 if there is a $\gamma > 0$ such that  %$G$ has property $(\mathcal{L})$ on $\ball_X(\bar{x}, \gamma) \times \ball_Y(\bar{y}, \gamma)$ for $\ell$ and $\gamma(\cdot) \equiv \gamma$, that is, 
 \begin{equation}\label{defApproxAubin}
 \lim\nolimits_{\varepsilon\downarrow 0} 	\dist\left(x,G^{-1}\big(\ball_Y(y',\varepsilon)\big)\right)\leq \ell \,\varrho(y,y')
\end{equation}
for each $y \in Y$, each $y'  \in \ball_Y(\bar{y}, \gamma)$, and each $x \in G^{-1}(y)\cap \ball_X(\bar{x}, \gamma)$ such that $ \ell \,\varrho(y,y') < \gamma$.
\end{itemize}
The \emph{modulus of almost openness} of $G$  around  $(\bar{x}, \bar{y})$,  denoted by $\overline{\sur}\,G(\bar{x}, \bar{y})$,
is the  supremum of the set of all $c > 0$ such that $G$ is almost open with the constant $c$ around $(\bar{x}, \bar{y})$. The \emph{modulus of almost metric regularity} of $G$  around  $(\bar{x}, \bar{y})$,  denoted by
$\overline{\reg}\,G(\bar{x}, \bar{y})$, is the  infimum of the set of all $\kappa > 0$ such that $G$ is almost metrically regular with the constant $\kappa$ around $(\bar{x}, \bar{y})$. The \emph{approximate Lipschitz modulus} of $G^{-1}$  around  $(\bar{y}, \bar{x})$, denoted by
$\overline{\lip}\, G^{-1}(\bar{y}, \bar{x})$, is the infimum of the set of all $\ell > 0$ such that $G$ has an almost pseudo-Lipschitz inverse with the constant $\ell$ around $(\bar{x}, \bar{y})$. 
\end{definition}

\begin{remark} \label{rmAlmostRegularSeq} \rm Fix any $x \in X$ and $y \in Y$ with $\dist(y, G(x)) < \infty$.
{\it Inequality \eqref{defApproxMetricRegularity} holds true if and only if for each $\varepsilon > 0$ there is a sequence $(y_k)_{k \in \mathbb{N}}$ converging to $y$ such that} 
\begin{equation}\label{defApproxMetricRegularityB}
  	\dist\big(x,G^{-1}(y_k)\big)\leq \kappa \, \big(\dist(y, G(x)) + \varepsilon\big) 
  	\quad \mbox{for each} \quad k \in \mathbb{N}.
\end{equation}
Assume that \eqref{defApproxMetricRegularity} holds true. Let  $\varepsilon > 0$ be arbitrary.  For each $k \in \mathbb{N}$ we have 
\begin{eqnarray*}
 \dist\left(x,G^{-1}\big(\ball_Y(y,\varepsilon/k)\big)\right) & \leq & \lim_{k \to \infty} 	\dist\left(x,G^{-1}\big(\ball_Y(y,\varepsilon/k)\big)\right) \leq   \kappa  \dist(y,G(x)) \\
 & < &   \kappa \, \big(\dist(y,G(x)) + \varepsilon\big).
\end{eqnarray*}
Thus there is  a sequence $(x_k)_{k \in \mathbb{N}}$ such that for each $k \in \mathbb{N}$ we have  $d(x,x_k) < \kappa \, \big(\dist(y,G(x)) + \varepsilon\big)$ and  $G(x_k)  \cap \ball_Y(y,\varepsilon/k) \neq \emptyset$, containing a point $y_k$, say. Hence $y_k \to y$ as $k \to \infty$ and 
$$
 \dist\big(x,G^{-1}(y_k)\big)\leq d(x,x_k)   < \kappa \, \big(\dist(y,G(x)) + \varepsilon\big) 
 \quad \mbox{for each} \quad k \in \mathbb{N}.
$$
On the other hand, pick any $\varepsilon > 0$. Find $(y_k)_{k \in \mathbb{N}}$ satisfying \eqref{defApproxMetricRegularityB} and converging to $y$. Let $k \in \mathbb{N}$ be such that $y_k \in \ball_Y(y,\varepsilon)$. Then 
$\dist\left(x,G^{-1}\big(\ball_Y(y,\varepsilon)\big)\right)\leq \dist\big(x,G^{-1}(y_k)\big)\leq \kappa \, \big(\dist(y, G(x)) + \varepsilon\big)
$.
Letting $\varepsilon \downarrow 0$, we get \eqref{defApproxMetricRegularity}.
\end{remark}

\begin{remark} \label{rmAlmostLip} \rm Fix any $x \in X$ and $y$, $y' \in Y$.
	{\it Inequality \eqref{defApproxAubin} holds true if and only if for each $\varepsilon > 0$ there is a sequence $(y_k)_{k \in \mathbb{N}}$ converging to $y'$ such that} 
	\begin{equation}\label{defApproxLipschitzB}
	\dist\big(x,G^{-1}(y_k)\big)\leq \ell \, \big(\varrho(y, y') + \varepsilon\big) 
	\quad \mbox{for each} \quad k \in \mathbb{N}.
	\end{equation}
     Assume that \eqref{defApproxAubin} holds true.  Let  $\varepsilon > 0$ be arbitrary.  For each $k \in \mathbb{N}$ we have 
\begin{eqnarray*}
	\dist\left(x,G^{-1}\big(\ball_Y(y',\varepsilon/k)\big)\right) & \leq &  \lim_{k \to \infty} 	\dist\left(x,G^{-1}\big(\ball_Y(y',\varepsilon/k)\big)\right) \leq  \ell \,  \varrho(y,y') \\
	&  < &  \ell \, \big( \varrho (y,y') + \varepsilon\big).
\end{eqnarray*}
	Thus there is  a sequence $(x_k)_{k \in \mathbb{N}}$ such that for each $k \in \mathbb{N}$ we have  $ d(x,x_k) < \ell \, \big( \varrho (y,y') + \varepsilon\big)$ and  $G(x_k) \cap \ball_Y(y',\varepsilon/k) \neq \emptyset$, containing a point $y_k$, say.  Hence $y_k \to y'$ as $k \to \infty$ and 
	$$
	\dist\big(x,G^{-1}(y_k)\big)\leq d(x,x_k)   < \ell \, \big( \varrho (y,y') + \varepsilon\big)
	 \quad \mbox{for each} \quad k \in \mathbb{N}.
    $$
    On the other hand, pick any $\varepsilon > 0$. Find  $(y_k)_{k \in \mathbb{N}}$ satisfying  \eqref{defApproxLipschitzB} and converging to $y'$. Let $k \in \mathbb{N}$ be such that $y_k \in \ball_Y(y',\varepsilon)$. Then 
	$\dist\left(x,G^{-1}\big(\ball_Y(y',\varepsilon)\big)\right)\leq \dist\big(x,G^{-1}(y_k)\big)\leq \ell \, \big( \varrho(y, y') + \varepsilon\big)
	$.
	Letting $\varepsilon \downarrow 0$, we get \eqref{defApproxAubin}. 
\end{remark}
%The mapping $g$  is said to have the {\it approximate Aubin property} around
%$\bar{x}$ if there are positive constants $\ell$ and $\gamma$ such that for each  $ x\in\ball_X(\bx,\gamma)$ and each $y\in\ball_Y(g(\bx),\gamma)$,  with  %$\ell \, \varrho(g(x),y) < \gamma$, we have 
%\begin{equation}\label{defApproxPseudoLipschitz}
%\lim_{\varepsilon\downarrow 0} 	\dist\left(g(x),g\big(\ball_X(x,\varepsilon)\big)\right)\leq \ell \, \varrho(g(x),y).
%\end{equation}
%The  infimum of the set of all $\ell > 0$ such that \eqref{defApproxPseudoLipschitz}  holds for some $\gamma>0$ is called the 
%{\it rate/modulus of the approximate Aubin property} of $g$  around  $\bar{x}$ and is denoted by
%$\overline{\lip}\,g(\bar{x})$.

As in the case of regularity we have the following equivalence:

\begin{proposition} \label{propApproxRegularityEquiv} Let  $(X,d)$ and $(Y, \varrho)$ be metric spaces. Consider a mapping  $G:X \rightrightarrows Y$ and a point $(\bar{x}, \bar{y}) \in \gph G$. Then
$
  \overline{\sur}\,G(\bar{x}, \bar{y}) \cdot \overline{\reg}\,G(\bar{x}, \bar{y}) = 1
$ 
and
$
\overline{\reg}\,G(\bar{x}, \bar{y}) =  \overline{\lip}\,G^{-1}(\bar{y}, \bar{x})
$.
\end{proposition}

\begin{proof} For each $c \in \big(0, \overline{\sur}\,G(\bar{x}, \bar{y})\big)$, if there is any, Proposition~\ref{propApproxRegularitiesEquiv} 
implies that  $1 = c \cdot \hbox{$\frac{1}{c}$} \geq c \cdot \overline{\reg}\,G(\bar{x}, \bar{y})$;  letting $c\uparrow \overline{\sur}\,G(\bar{x}, \bar{y})$, we get $1 \geq \overline{\sur}\,G(\bar{x}, \bar{y})\cdot \overline{\reg}\,G(\bar{x}, \bar{y})$. By the convention $0 \cdot \infty = 1$, the last inequality  also holds when $\overline{\sur}\,G(\bar{x}, \bar{y})= 0$. 
Similarly, for  each $\kappa > \overline{\reg}\,G(\bar{x}, \bar{y})$, if there is any, Proposition~\ref{propApproxRegularitiesEquiv} 
implies that  $1 =  \hbox{$\frac{1}{\kappa}$} \cdot \kappa \leq \overline{\sur}\,G(\bar{x}, \bar{y}) \cdot \kappa$; letting  $\kappa \downarrow \overline{\reg}\,G(\bar{x}, \bar{y})$, we get $1 \leq \overline{\sur}\,G(\bar{x}, \bar{y}) \cdot \overline{\reg}\,G(\bar{x}, \bar{y})$. By our convention, the last inequality  also holds when $\overline{\reg}\,G(\bar{x}, \bar{y})= \infty$. We proved the first relation; the latter one is obvious by Proposition~\ref{propApproxRegularitiesEquiv}.
\end{proof}

Taking a neighborhood of the reference point as $U$ and $V$ being singleton containing the reference point only, we  get the notion of \emph{almost subregularity}:   
\begin{definition} \label{defApproxSubregularity}
Let  $(X,d)$ and $(Y, \varrho)$ be metric spaces. Consider a mapping  $G:X \rightrightarrows Y$ and a point $(\bar{x}, \bar{y}) \in \gph G$. Then $G$ is said to 
\begin{itemize} 
   \item[$(i)$] be \emph{almost pseudo-open with a constant $c>0$  at} $(\bar{x}, \bar{y})$
 if there is a $\gamma > 0$ such that %$G$ has property $(\mathcal{O})$ on $\ball_X(\bar{x}, \gamma) \times \{\bar{y}\}$ for $c$ and $\gamma(\cdot) \equiv \gamma$, that is, 
%\begin{equation}\label{defApproxPseudoOpen}
 $ \bar{y}\in \overline{G\left(\ball_X(x,t)\right)} $
%\end{equation}
 for each  $x\in\ball_X(\bx,\gamma)$ and each $t\in (0,\gamma)$ with $G(x) \cap \ball_Y(\bar{y},ct) \neq \emptyset$;
 \item[$(ii)$] be \emph{almost metrically subregular with a constant $\kappa>0$ at} $(\bar{x}, \bar{y})$ 
 if there is a $\gamma > 0$ such that % $G$ has property $(\mathcal{R})$ on $\ball_X(\bar{x}, \gamma) \times \{\bar{y} \}$ for $\kappa$ and $\gamma(\cdot) \equiv \gamma$, that is,  
%\begin{equation}\label{defApproxMetricSubregularity}
 $ \lim\nolimits_{\varepsilon\downarrow 0} 	\dist\left(x,G^{-1}\big(\ball_Y(\bar{y},\varepsilon)\big)\right) \leq \kappa \, \dist (\bar{y},G(x)) $
%\end{equation}
 for each $x\in\ball_X(\bx,\gamma)$  with $ \kappa \, \dist(\bar{y}, G(x)) < \gamma$;
   \item[$(iii)$] have an \emph{almost calm inverse with a constant $\ell > 0$ at} $(\bar{x}, \bar{y})$
 if there is a $\gamma > 0$ such that % $G$ has property $(\mathcal{L})$ on $\ball_X(\bar{x}, \gamma) \times \{ \bar{y}\}$ for $\ell$ and $\gamma(\cdot) \equiv \gamma$, that is,   
%\begin{equation}\label{defApproxCalm}
 $ \lim\nolimits_{\varepsilon\downarrow 0} 	\dist\left(x,G^{-1}\big(\ball_Y(\by,\varepsilon)\big)\right)\leq \ell \, \varrho(\by,y)$
%\end{equation}
for each $ y \in Y$ and each $x  \in G^{-1}(y) \cap \ball_X(\bar{x}, \gamma)$ with $ \ell \, \varrho (\by, y) < \gamma$.
\end{itemize}
The  \emph{modulus of almost pseudo-openness} of $G$  at  $(\bar{x}, \bar{y})$, denoted by
$\overline{\psopen}\,G(\bar{x}, \bar{y})$, is the  supremum of the set of all $c > 0$ such that $G$ is almost pseudo-open with the constant $c$ at $(\bar{x},\bar{y})$.
The \emph{modulus of almost metric subregularity} of $G$  at  $(\bar{x}, \bar{y})$, denoted by
$\overline{\subreg}\,G(\bar{x}, \bar{y})$, is the  infimum of the set of all $\kappa > 0$ such that  $G$ is almost metrically subregular with the constant $\kappa$ at $(\bar{x},\bar{y})$.
The \emph{modulus  of almost calmness} of $G^{-1}$  at  $(\bar{y}, \bar{x})$, denoted by
$\overline{\calm}\,G^{-1}(\bar{y}, \bar{x})$, is the  infimum of the set of all $\ell > 0$ such that  $G$ has an almost calm inverse with the constant $\ell$ at $(\bar{x},\bar{y})$.
\end{definition}

As in the case of subregularity,  using Proposition~\ref{propApproxRegularitiesEquiv} and the same steps from the proof of Proposition~\ref{propApproxRegularityEquiv}, we get:

\begin{proposition} \label{propApproxSubegularityEquiv} Let  $(X,d)$ and $(Y, \varrho)$ be metric spaces. Consider a mapping  $G:X \rightrightarrows Y$ and a point $(\bar{x}, \bar{y}) \in \gph G$. Then
$ 
 \overline{\psopen}\,G(\bar{x}, \bar{y}) \cdot \overline{\subreg}\,G(\bar{x}, \bar{y}) = 1
$ and
$ 
 \overline{\subreg}\,G(\bar{x}, \bar{y}) =  \overline{\calm}\,G^{-1}(\bar{y}, \bar{x})
$.
\end{proposition}

On the other hand, if $U$ is a singleton containing the reference point only and $V$ is a neighborhood of the reference point, we  get the notion of \emph{almost 
semiregularity}, see \cite{CFK2019} for the historical overview of semiregularity property together with a list of different names used in the literature.   
\begin{definition} \label{defApproxSemiregularity}
Let  $(X,d)$ and $(Y, \varrho)$ be metric spaces. Consider a mapping  $G:X \rightrightarrows Y$ and a point $(\bar{x}, \bar{y}) \in \gph G$. Then $G$ is said to  
\begin{itemize} 
   \item[$(i)$] be \emph{almost open with a constant $c >0$ at} $(\bar{x}, \bar{y})$
 if there is a $\gamma>0$ such that % $G$ has property $(\mathcal{O})$ on $\{\bar{x}\} \times \ball_Y(\bar{y}, \gamma)$ for  $c$ and $\gamma(\cdot) \equiv \gamma$, that is,  
%\begin{equation}\label{defApproxOpenAt}
$\ball_Y(y,ct) \cap \ball_Y(\bar{y}, \gamma)  \subset \overline{G\big(\ball_X(\bar{x},t)\big)}$
%\end{equation}
 for each $ y \in G(\bar{x})$  and each $  t  \in (0,\gamma)$;
 \item[$(ii)$] be \emph{almost metrically semiregular  with a constant $\kappa >0$ at} $(\bar{x}, \bar{y})$ if  there is a $\gamma>0$ such that 
 %$G$ has property $(\mathcal{R})$ on $\{\bar{x} \} \times \ball_Y(\bar{y}, \gamma)$ for $\kappa$ and $\gamma(\cdot) \equiv \gamma$, that is, 
%\begin{equation} \label{defApproxMetricSemiregularity}
$\lim\nolimits_{\varepsilon\downarrow 0} 	\dist\left(\bar{x},G^{-1}\big(\ball_Y(y,\varepsilon)\big)\right)\leq \kappa \, \dist(y, G(\bar{x}))$
%\end{equation}
for each $y  \in  \ball_Y(\bar{y},\gamma)$ with $ \kappa  \dist(y, G(\bar{x})) < \gamma$; 
   \item[$(iii)$] have an \emph{almost inner-calm inverse  with a constant $\ell >0$ at} $(\bar{x}, \bar{y})$
  if there is a $\gamma>0$ such that %  $G$ has property $(\mathcal{L})$ on $\{\bar{x}\} \times \ball_Y(\bar{y}, \gamma)$ for $\ell$ and $\gamma(\cdot) \equiv \gamma$, that is,   
%\begin{equation}\label{defApproxRecede}
$ \lim\nolimits_{\varepsilon\downarrow 0} 	\dist\left(\bx,G^{-1}\big(\ball_Y(y',\varepsilon)\big)\right)\leq \ell \, \varrho (y,y')$
%\end{equation}
 for each $y \in G(\bx)$ and each $y' \in \ball_Y(\bar{y}, \gamma)$  with $ \ell \, \varrho (y,y') < \gamma$.
\end{itemize}
The  \emph{modulus of almost openness} of $G$  at  $(\bar{x}, \bar{y})$, denoted by
$\overline{\lopen}\,G(\bar{x}, \bar{y})$, is the  supremum of the set of all $c > 0$ such that $G$ is almost open with the constant $c$ at $(\bar{x},\bar{y})$.
The \emph{modulus of almost metric semiregularity} of $G$  at  $(\bar{x}, \bar{y})$, denoted by
$\overline{\semireg}\,G(\bar{x}, \bar{y})$, is the  infimum of the set of all $\kappa > 0$ such that  $G$ is almost metrically semiregular with the constant $\kappa$ at $(\bar{x},\bar{y})$.
The \emph{modulus  of almost inner-calmness} of $G^{-1}$  at  $(\bar{y}, \bar{x})$, denoted by
$\overline{\recess}\,G^{-1}(\bar{y}, \bar{x})$, is the  infimum of the set of all $\ell > 0$ such that  $G$ has an almost inner-calm inverse with the constant $\ell$ at $(\bar{x},\bar{y})$.
\end{definition}

As in the case of semiregularity,  using Proposition~\ref{propApproxRegularitiesEquiv} and  the same steps from the proof of Proposition~\ref{propApproxRegularityEquiv}, we get:

\begin{proposition} \label{propApproxSemiegularityEquiv} Let  $(X,d)$ and $(Y, \varrho)$ be metric spaces. Consider a mapping  $G:X \rightrightarrows Y$ and a point $(\bar{x}, \bar{y}) \in \gph G$. Then
$ 
 \overline{\lopen}\,G(\bar{x}, \bar{y}) \cdot \overline{\semireg}\,G(\bar{x}, \bar{y}) = 1
$ and
$ 
 \overline{\semireg}\,G(\bar{x}, \bar{y}) =  \overline{\recess}\,G^{-1}(\bar{y}, \bar{x})
$.
\end{proposition}

\begin{remark}   \rm For a single-valued mapping $g: X \to Y$, we omit the point $\bar{y}=g(\bar{x})$ in all the preceding definitions; and if any of the  moduli is independent of the reference point, for example, in case of a linear operator acting in normed spaces, we omit $\bar{x}$ as well; that is, we write    $\overline{\sur} g(\bar{x})$, $\overline{\reg} \, g(\bar{x})$, etc.; or even   $\overline{\sur} g$, $\overline{\reg} \, g$, etc.
\end{remark}

\begin{remark} \label{rmkClosedBalls} \rm
The value of  $\overline{\sur}  \, G(\bar{x}, \bar{y})$ remains the same if, one replaces \eqref{defApproxOpen} by  
%\begin{equation}\label{defApproxOpenB}
$\ball_Y[y,ct] \cap \ball_Y(\bar{y}, \gamma)  \subset \overline{G\big(\ball_X[x,t]\big)}.$
%  \mbox{whenever} \  x \in \ball_X(\bar{x}, \gamma), \ y\in G(x),   t  \in (0,\gamma).
%\end{equation}
Of course, the definitions of  $\overline{\psopen} \, G(\bar{x}, \bar{y})$  and  $\overline{\lopen} \, G(\bar{x}, \bar{y})$ can be modified similarly. The use of the open balls provides the following advantage:
{\it Given a $c > 0$, if a~mapping $G: X \rightrightarrows Y$ has property $(\mathcal{O})$ on $U \times V$ for  $\gamma$ and {\tt each} $c' \in (0,c)$, then $G$ has property $(\mathcal{O})$ on $U\times V$ for  $\gamma$ and  $c$.}
\end{remark}

\section{Ioffe-type criteria} \label{scIoffeCriteria}

In this section, we focus on almost regularity  of a~single-valued mapping first. This allows us to display clearly the main ideas avoiding seemingly more general proofs for set-valued mappings popular nowadays. 

\begin{proposition}\label{propAlmostRegSubreg} Let $(X,d)$ and $(Y,\varrho)$ be metric spaces. Consider a $c> 0$, non-empty sets $U \subset X$ and $V \subset Y$, a function $\gamma: X \to [0, \infty]$  being  defined and not identically zero on $U$, and a mapping  $g: X \to Y$ defined on the whole $X$.  Assume that for each $y \in V$ there is a $\lambda = \lambda(y) \in (0,1)$ such that for each $u \in X$ satisfying 
\begin{equation}\label{eqAlmostConstraint}
0  < \varrho(g(u),y) \leq \varrho(g(x),y) -  c \, d(u,x)
 \quad \mbox{and}  \quad 
 \varrho(g(x),y) <   c \gamma(x)
\end{equation}
for some   $x \in U$, there is a point  $u'\in X$ %, better than $u$, 
such that 
\begin{equation}\label{eqAlmostBetter}
 \varrho(g(u'),y) + c\,d(u,u')   \leq \lambda \, \varrho(g(u),y).
\end{equation}
Then $g$ has property $(\mathcal{O})$ on $U \times V$ for $c$ and $\gamma$.
\end{proposition}

\begin{proof} Let $x \in U$ and $t  \in (0,\gamma(x))$ be arbitrary. Pick any  $y \in \ball_Y(g(x),ct) \cap  V$. Find a $\lambda \in (0,1)$ such that the assumption containing \eqref{eqAlmostBetter} holds true.
Let $\varepsilon  >0 $ be arbitrary. We are about to find a $u \in \ball_X(x,t)$ such that $\varrho(g(u),y) \leq \varepsilon$. If $\varrho(g(x),y) \leq \varepsilon$, then $u:= x$ satisfies the conclusion. Assume that this is not the case.  
	As $(1-\lambda) \varepsilon < \varepsilon < \varrho (g({x}), y) < ct < \infty$, Corollary~\ref{corApproxEkeland}, with $\varphi:=\varrho (g(\cdot), y)$, $\eta := c d$,  and $\varepsilon := (1-\lambda)\varepsilon$, yields a point $u \in X$ such that $	\varrho(g(u),y)  + c \, d(u, x) \le \varrho(g({x}),y)  $
and
	\begin{equation} \label{eqApproxBetter10}
	  \varrho(g(u'),y) +  c \, d(u',u) > \varrho(g(u),y) - (1-\lambda) \varepsilon \quad \mbox{for every} \quad u' \in X.
	\end{equation}
	The first inequality implies that $ c \, d(u, {x}) \leq \varrho(g({x}),y) < c t$ $(< c \gamma(x)$). Thus  $u \in  \ball_X(x,t)$.   Assume that  $\varrho(g(u),y) >  \varepsilon$.
   Find a  $u' \in X$ satisfying \eqref{eqAlmostBetter}. Then   	
    \begin{eqnarray*}
       c\,d(u,u') & \overset{\eqref{eqAlmostBetter}}{\leq} &  \lambda \, \varrho(g(u),y)  -\varrho(g(u'),y) <   \varrho(g(u),y)  - (1-\lambda) \varepsilon  -\varrho(g(u'),y)  \\ 
      & \overset{\eqref{eqApproxBetter10}}{<} &  \varrho(g(u'),y) + c\,d(u',u) - \varrho(g(u'),y)  =  c\,d(u',u), 
     \end{eqnarray*}
  a~contradiction. Consequently, $\varrho(g(u),y) \leq  \varepsilon$. Hence $\dist \left(y, g\big(\ball_X(x,t)\big)\right) \leq \varrho(y,g(u)) \leq  \varepsilon$. Letting $\varepsilon \downarrow 0$, we conclude that $y \in \overline{g\big(\ball_X(x,t)\big)}$. Thus $\ball_Y(g(x),ct) \cap  V \subset \overline{g\big(\ball_X(x,t)\big)}$. 	
\end{proof}

\begin{remark} \label{rmkGamma} \rm We do not need any additional assumption on $\gamma$ except for the ones ensuring that property  
$(\mathcal{O})$ is meaningful. Usually, the constraint \eqref{eqAlmostConstraint} is stated in a weaker form. This excludes obtaining applicable conditions for (almost) semiregularity and, moreover, we need additional properties of $\gamma$. Assume that $\gamma$ is Lipschitz continuous  on $\widetilde{U}:= \cup_{x\in U} \ball_X(x, \gamma(x))$ with the constant $1$. Given $u \in X$, $x \in U$, and $y \in  V$ satisfying \eqref{eqAlmostConstraint}, we have 
$$
  \varrho(g(u),y) \leq  \varrho(g(x),y) -  c \, d(u,x) <  c \, (\gamma(x) -  d(u,x)) \leq c \gamma(u);
$$ 
so, in particular, $d(u,x) < \gamma(x)$. 
Thus the assumption containing \eqref{eqAlmostBetter} can be replaced by a less precise and stronger one: {\it Assume that there is a $\lambda \in (0,1)$ such that for each $u \in \widetilde{U}$ and each $y \in V$, with  $0 < \varrho(g(u),y) <  c \gamma(u)$, there is a point  $u'\in X$ satisfying \eqref{eqAlmostBetter}.} If $U$ is open and $\gamma(x) := \dist(x, X\setminus U)$, $x \in X$, then $\gamma$ is positive on $U$ and we arrive at \emph{Milyutin almost regularity}. 
\end{remark}

An approximate version of Proposition~\ref{propFabianPreiss} reads as:

\begin{corollary}\label{corAlmostFabianPreiss} Let $(X,d)$ and $(Y,\varrho)$ be metric spaces. Consider  $c> 0$ and $r > 0$, a~point $\bar{x} \in X$,  and a mapping  $g: X \to Y$ defined on the whole $X$.  Assume that  there is a $\lambda  \in (0,1)$ such that for each $u \in \ball_X(\bar{x},r)$ and each $y \in \ball_Y(g(\bar{x}),cr)$, with $ 0 < \varrho(g(u),y) < cr$,
there is a point  $u'\in X$ %, better than $u$, 
satisfying  \eqref{eqAlmostBetter}. Then $g$ has property $(\mathcal{O})$ on $\ball_X(\bar{x},r) \times \ball_Y(g(\bar{x}),cr)$ for $c$ and $\gamma(x):= r - d(x, \bar{x})$, $x \in \ball_X(\bar{x},r)$. In particular, $g$ is almost regular around $\bar{x}$ with the constant $c$. 
\end{corollary}

\begin{proof}
Let $U:= \ball_X(\bar{x},r)$ and  $V:=\ball_Y(g(\bar{x}),cr)$.  Then $\widetilde{U}:=\cup_{x\in U} \ball_X(x, \gamma(x)) = \ball_X(\bar{x},r) = U$ and $\gamma(u) \leq r$ for each $u \in U$. In view of Remark~\ref{rmkGamma}, Proposition~\ref{propAlmostRegSubreg} implies the first conclusion. Let $\delta := \min\{r/2, cr\}$. For each $x \in \ball_X(\bar{x},\delta)$ and each $t \in (0, \delta)$ we have $t <  r/2 < r - d(x,\bar{x}) = \gamma(x)$ 
and, therefore,  
$$
\ball_Y(g(x),ct) \cap \ball_Y(g(\bar{x}), \delta)  \subset \ball_Y(g(x),ct) \cap \ball_Y(g(\bar{x}),cr) \subset  \overline{g\big(\ball_X(x,t)\big)}.
$$     
\end{proof}

\begin{remark} \label{rmkBetter} \rm 
Clearly,   \eqref{eqAlmostBetter} implies \eqref{eqBetter}.  Hence if, in addition, the space $X$ is complete and $g$ is continuous then we get full regularity, cf. Proposition~\ref{propFabianPreiss}. In \cite{IoffeBook}, \eqref{eqBetter} with a non-strict inequality is used under the additional assumption that the ``better'' point $u'$ is distinct from $u$. Note that \eqref{eqBetter} yields this automatically.    
\end{remark}
 
A statement for set-valued mappings,  \emph{equivalent} to Proposition~\ref{propAlmostRegSubreg}, reads as: 

\begin{proposition}\label{propAlmostRegSetvalued2} Let $(X,d)$ and $(Y,\varrho)$ be metric spaces. Consider  $c> 0$ and $\alpha \in (0,1/c)$, non-empty sets $U\subset X$ and $V\subset Y$, a function $\gamma: X\rightarrow [0,\infty]$ being defined and not identically zero on $U$, and a set-valued mapping  $G: X \rightrightarrows Y$. 
Assume that  each $y \in V$ there is a $\lambda = \lambda(y) \in (0, 1)$ such that for each $(u,v) \in \gph G$ satisfying 
$$
  0 < \varrho(v,y) \leq \varrho(z,y) -  c\,\max\{d(u,x), \alpha \varrho(v,z)\}
  \quad \mbox{and} \quad 
  \varrho(z,y) < c\gamma(x)
$$ 
for some $x \in U$ and $z \in G(x)$, there is a~pair  $(u',v') \in \gph G$ %, better than $(u,v)$, 
such that
		\begin{eqnarray}\label{eqAlmostBetterReg2}
			\varrho(v',y) + c\,\max\{d(u,u'), \alpha \varrho(v,v')\} \leq \lambda \, \varrho(v,y).
		\end{eqnarray}
Then $G$ has property $(\mathcal{O})$ on $U \times V$ for $c$ and $\gamma$.
\end{proposition}

\begin{proof} Let $\omega$, $g$,  and $\widetilde{\gamma}$ be as in Remark~\ref{rmkProjection}. In view of this remark, it suffices to observe that all the assumptions of Proposition~\ref{propAlmostRegSubreg}, with $(X,d) := (\gph G, \omega)$, $U := (U \times Y)\cap \gph G$, $\gamma := \widetilde{\gamma}$, are satisfied; which is obvious.
\end{proof}

A statement for set-valued mappings,  \emph{equivalent} to Corollary~\ref{corAlmostFabianPreiss}, reads as: 

\begin{corollary}\label{corAlmostRegSetvalued2} Let $(X,d)$ and $(Y,\varrho)$ be metric spaces. Consider  $c> 0$, $r > 0$,  and $\alpha \in (0,1/c)$, a point $(\bar{x}, \bar{y}) \in X \times Y$, and a set-valued mapping  $G: X \rightrightarrows Y$ with $\by \in G(\bx)$. 
Assume that there is a $\lambda \in (0, 1)$ such that  for each $u \in \ball_X(\bar{x},r)$, each $v  \in G(u) \cap \ball_Y(\bar{y},r/\alpha)$, and each $y \in \ball_Y(\bar{y},cr)$, with $0 < \varrho(v,y) <  cr $, there is a~pair  $(u',v') \in \gph G$ %, better than $(u,v)$, 
satisfying \eqref{eqAlmostBetterReg2}. 
Then $G$ has property $(\mathcal{O})$ on $\ball_X(\bar{x},r) \times \ball_Y(\bar{y},cr)$ for $c$ and  $\gamma(x):= r - d(x, \bar{x})$, $x \in \ball_X(\bar{x},r)$. In particular, $G$ is almost regular around $(\bar{x}, \bar{y})$ with the constant $c$. 
\end{corollary}

Another frequent choice is $\gamma \equiv r$, for some  $r > 0$. As in the case of regularity, if, in addition, $U \times V$ is a neighborhood of the reference point, then this third parameter $r$ can be neglected in exchange for a possibly smaller neighborhood, cf.  \cite[Exercise 2.12]{IoffeBook}. We illustrate this on another form of property $(\mathcal{R})$, equivalent to property $(\mathcal{O})$,  which have become popular in the last years: 

   \begin{remark} \label{rmkMetricRegularity}  \rm Given a point $(\bar{x},\bar{y})\in X\times Y$ and a  mapping $G:X\tto Y$ with $\bar{y}\in G(\bar{x})$, assume that  there are positive constants  $a$, $b$, $c$, and $r$ such that 
  \begin{equation} \label{defApproxMetricRegularityC}
  c \, \lim\nolimits_{\varepsilon\downarrow 0} 	 \dist\left(x,G^{-1}\big(\ball_Y(y,\varepsilon)\big)\right)\leq  \dist(y, G(x))
\end{equation}
for each $x \in\ball_X(\bx,a)$ and each $y \in \ball_Y(\bar{y},b)$ with  $\dist(y, G(x)) < c r$.   Let  $\beta:= \min\{a, b, cr/(1+c)\}$. Then \eqref{defApproxMetricRegularityC} holds for each $(x,y) \in\ball_X(\bx,\beta) \times \ball_Y(\bar{y},\beta)$.  
   	Indeed,  fix any such $(x,y)$.  Pick an arbitrary  $v\in G(x)$ (if there is any). Assume that $ \varrho(y,v)\geq cr$. As $\dist(y,G(\bar{x})) \leq \varrho (y, \bar{y}) < \beta < cr$ and $\beta \leq b$, \eqref{defApproxMetricRegularityC} with $x := \bar{x}$, implies that
   	\begin{eqnarray*}
   		c \, \lim\nolimits_{\varepsilon\downarrow 0} \dist\left(x,G^{-1}\big(\ball_Y(y,\varepsilon)\big)\right) & \leq &   c \, d(x,\bar{x})  + 	c \, \lim\nolimits_{\varepsilon\downarrow 0}  \dist\left(\bx,G^{-1}\big(\ball_Y(y,\varepsilon)\big)\right) \\
   		& \leq  & c \, d(x,\bar{x}) + \varrho(y,\bar{y}) <  c \beta + \beta   \leq  c  r  \leq  \varrho(y,v).
   	\end{eqnarray*}
   	If $\varrho(y,v)< cr$,  the same inequality follows from \eqref{defApproxMetricRegularityC} directly. As $v\in G(x)$ is arbitrary, we are done. 
  \end{remark}

\section{Perturbation stability}

Applying Corollary~\ref{corAlmostRegSetvalued2}, we get a local approximate Lyusternik-Graves theorem, cf. \cite[Theorem~5E.1]{DontchevBook}, in not necessarily complete spaces.

\begin{theorem} \label{thmApproxL-Gset}
	Let $(X,d)$ be a metric space and $(Y, \varrho)$ be a linear metric space with a~translation-invariant metric.  Consider a point $(\bx,\bz) \in X\times Y$, a set-valued mapping $F:X\tto Y$ with $\bz\in F(\bx)$, and a single-valued mapping $h: X \to Y$  defined and Lipschitz continuous around $\bx$.   
	Then  
	\begin{equation} \label{eqApproxSurfFset}
		\overline{\sur}(F+h)(\bx, \bz + h(\bx)) \geq  \overline{\sur}\,F(\bx, \bz) - \lip h (\bx).
	\end{equation}
\end{theorem}
\begin{proof}
	 Assume that $\overline{\sur}\,F(\bx, \bz) > \lip h (\bx)$ because otherwise there is nothing to prove. Fix constants  $c$, $c'$, and $\ell$ such that 
	$$
	\lip h(\bx) < \ell < c' < c< \overline{\sur}\,F(\bx, \bz).
	$$
	Find a $\gamma  > 0$ such that for  each $u \in \ball_X(\bx,\gamma)$,  each $w\in F(u)$, and each $t \in (0,\gamma)$ we have
	\begin{equation}
		\label{eqApproxLinOpennesSet}
		\ball_Y[w,ct] \cap  \ball_Y(\bz, \gamma) \subset \overline{F(\ball_X[u,t])};   \qquad   \qquad (\mbox{see Remark~\ref{rmkClosedBalls}}). 
	\end{equation}
		Pick an $r \in \big(0,  \min\{\gamma, \gamma/(3c)\}\big)$ such that 
	\begin{equation} \label{eqApproxSum3Set} % h(u) \in \ball_Y(h(\bx), \gamma) \quad \mbox{and} \quad 
		\varrho(h(u), h(u')) \leq \ell \, d(u,u') \quad \mbox{for all} \quad u, u' \in \ball_X(\bx, 2 r ).  
	\end{equation}
	Let $\alpha:=1/(2c)$ and 	$\lambda := (c + c')/(2c)$. Then $\lambda \in (0,1)$ and $(c'-\ell)\alpha<1$. Let $G := F + h$.  Pick arbitrary  $u \in \ball_X(\bx,r)$, 
	$v \in G(u) \cap \ball_Y(\bz+h(\bx),2cr)$, 
	and $y\in\ball_Y(\bz+h(\bx),(c'-\ell)r)$ with $0<\varrho(v,y)<(c'-\ell)r$. We shall find a pair $(u',v') \in \gph G$ such that 
	\begin{equation} \label{eqBetterPerturb} 
		\varrho(v', y)+(c' - \ell) \max \lbrace d(u,u'), 	\alpha \varrho(v, v')  \rbrace < \lambda \varrho (v, y). 
	\end{equation}
	Clearly,  $u\in\ball_X(\bx,\gamma)$. Let  $w:=v-h(u) \in F(u)$ and $t:=\varrho(v,y)/c$.
	Note that $0<t<(c'-\ell)r/c<r < \gamma$. Further, using \eqref{eqApproxSum3Set},  we get that  
	\begin{eqnarray*}
		\varrho(w,\bz) &=&\varrho(v-h(u)+h(\bx),\bz+h(\bx))\leq \varrho(v-h(u)+h(\bx),v)+\varrho(v,\bz+h(\bx))\\
		&=& \varrho(h(\bx),h(u))+\varrho(v,\bz+h(\bx)) \leq \varrho(h(\bx),h(u))+ \varrho(v,y) + \varrho(y,\bz+h(\bx)) \\
		&  <  & \ell r +  (c'-\ell) r + (c'-\ell) r < 2c'r.
	\end{eqnarray*}
	Thus  $y-h(u)\in\ball_Y[w,ct]\cap \ball_Y(\bz,\gamma)$ because $\varrho(y-h(u),w)=\varrho(y,v)=ct$ and
	\begin{eqnarray*}
		\varrho(y-h(u),\bz)\leq \varrho(y-h(u),w)+\varrho(w,\bz)<ct+ 2c'r< 3cr < \gamma.
	\end{eqnarray*}
	By  \eqref{eqApproxLinOpennesSet}, there are $u'\in\ball_X[u,t]$ and $w'\in F(u')$ such that $\varrho(y-h(u),w')<(c-c')t/2$. Then
	$d(u',\bx)\leq d(u',u)+d(u,\bx)<t +r<2r$.
	By \eqref{eqApproxSum3Set}, we have $\varrho(h(u),h(u'))\leq  \ell t$. Let $v':=w'+h(u') \in G(u')$. 
	Then 
	\begin{eqnarray}
		\varrho(v', y) & = & \varrho(w'+h(u'), y)  \leq \varrho(w'+h(u'),w' + h(u)) + \varrho(w'+ h(u),y) \notag\\
		&  = &  \varrho(h(u),h(u')) +  \varrho(y -  h(u),w') < \ell t + (c - c')t/2 \label{eqA001} \\
		& = &    (c+c')t/2  - (c' - \ell) t  = \lambda \,   \varrho (v, y) - (c' - \ell) t.  \label{eqA003}
	\end{eqnarray}
    Inequality in \eqref{eqA001} reveals  that 
	\begin{eqnarray*}
		\varrho(v, v')&\leq & \varrho(v,y)+ \varrho(y, v')   < c t + \ell t +(c-c')t/2 
		=  (3c/2 - c'/2 + \ell) t \\ 
		& < & (3c/2 + c'/2) t  < 2ct = t/\alpha.
	\end{eqnarray*}
    Hence $\alpha  \varrho(v, v') < t$; and as $d(u,u') \leq t$,  \eqref{eqA003} implies \eqref{eqBetterPerturb}.
    Corollary~\ref{corAlmostRegSetvalued2}, with  $c := c' - \ell$ and $\by := \bz + h(\bx)$,   yields that $\overline{\sur} \, G (\bx, \bz + h(\bx)) \geq c'- \ell$. Letting $c'\uparrow \overline{\sur} \, F(\bx, \bz)$ and $\ell \downarrow \lip h(\bx)$ we finish the proof.
\end{proof}

Theorem~\ref{thmApproxL-Gset} implies the following Graves-type result:
\begin{theorem} \label{thmApproxGraves}
Let $(X, \|\cdot\|)$ and $(Y, \|\cdot\|)$ be normed spaces. Consider a point $\bar{x} \in X$ and mappings $f$, $g: X \to Y$ both defined around $\bar{x}$ and such that $\lip(f-g)(\bar{x}) = 0$.  Then $\overline{\sur} f(\bx) = \overline{\sur}  g(\bx)$. 
\end{theorem}  

\begin{proof}   
 Theorem~\ref{thmApproxL-Gset}, with $F := f$ and $h:=g-f$; and $F := g$ and $h := f - g$, respectively, implies that $\overline{\sur} g(\bx)  \geq \overline{\sur} f(\bar{x}) -  \lip(g-f)(\bar{x}) = \overline{\sur} f(\bar{x})$; and $\overline{\sur} f(\bar{x}) \geq \overline{\sur} g(\bx) - \lip(f-g)(\bar{x}) = \overline{\sur} g(\bx)$.  
\end{proof}  

Now, we consider set-valued perturbations as in \cite[Theorems 3.1 and 3.10]{CR2023}. In contrast to \cite{CR2023},  we employ neither the epigraphical multifunction nor the distance to the graph of the sum in the proof.

\begin{theorem} \label{thmApproxL-Gsetset}
	Let $(X,d)$ be a metric space and $(Y,\varrho)$ be a linear metric space with a translation-invariant metric. Consider  $ c$, $\ell \in (0, \infty)$ such that  $ \ell < c$ and  $a$, $b$, $r$, $\delta \in (0,\infty]$,   a point $(\bar{x}, \bar{z}, \bar{w}) \in X \times Y \times Y$, and set-valued mappings $F$, $H:X\tto Y$ with $\bar{z}\in F(\bar{x})$ and $\bar{w} \in H(\bar{x})$. Assume that 
	\begin{itemize}
		\item[$(A)$] % $F$ has property $(\mathcal{O})$ on $\ball_{{X}}(\bx,a') \times \ball_Y(\bz, \delta)$ for $c'$ and $\gamma$; \\ 
		$\ball_Y[z, c t] \cap\ball_Y(\bz, c (a   +  2r) + b+\delta) \subset  \overline{F(\ball_X[u,t])}$ for each $u \in \ball_{{X}}(\bx, a + r)$, each $z \in F(u) \cap \ball_{Y}(\bz, c(a+r) + b + \delta)$ 
		and each $t\in (0, r)$;
		\item[$(B)$]  $H(u) \cap \ball_Y(\bar{w}, (a + r) \ell+\delta) \subset H(u') +  \ball_Y[0,\ell \, d(u,u')]$  for each $u$, $u' \in \ball_{{X}}(\bar{x}, a + 2r)$; 
		\item[$(C)$] for each $u \in \ball_{{X}}(\bar{x}, a+r)$ and each $v \in (F+H)(u) \cap \ball_{{Y}}(\bar{z} + \bar{w}, b + cr  )$ there is a $w \in H(u) \cap \ball_Y(\bar{w},(a + r) \ell + \delta)$ such that $v \in F(u) + w$.
	\end{itemize}
   Then $F+H$ has property $(\mathcal{O})$ on $\ball_{{X}}(\bx,a)\times \ball_{Y}(\bz+\bw, b)$ for $c-\ell$ and $\gamma \equiv r$.
\end{theorem}
\begin{proof} Pick an arbitrary $c' \in (\ell,c)$.  
 	Let $\alpha:=1/(2c)$ and $\lambda:=(c+c')/(2c)$. Then $\lambda \in (0,1)$ and $(c'-\ell)\alpha<1$. Let $G := F + H$.  Pick arbitrary  $u \in \ball_X(\bx,a + r)$, 
	$v \in G(u)$, and $y\in\ball_Y(\bz+\bar{w},b)$ with 
$0<\varrho(v,y)<(c'-\ell)r$. We shall find a pair $(u',v') \in \gph G$ such that \eqref{eqBetterPerturb} holds true.

 	As $\varrho(v, \bar{z} + \bar{w}) \leq \varrho(v,y) + \varrho(y, \bar{z} + \bar{w}) < (c' - \ell) r + b $, using (C), there is a~$w \in H(u) \cap \ball_Y(\bar{w},(a + r) \ell+\delta)$ such that $v \in F(u) + w$. 
Let  $z:=v-w \in F(u)$ and $t:=\varrho(v,y)/c$.
	Note that $0<t<(c'-\ell)r/c\leq r$  and 
	\begin{eqnarray*}
	 \varrho(z, \bz) &=& \varrho(v-w,\bz) \leq  \varrho(v-w,v-\bar{w})+\varrho(v-\bar{w},\bz) =
	  \varrho(w,\bar{w})+\varrho(v,\bz+\bw) \\
	  &<& (a + r) \ell +\delta + (c'-\ell) r + b \leq c(a + r)  + b +\delta.
	\end{eqnarray*}	
	Thus $y - w \in \ball_Y[z,ct]\cap \ball_Y(\bz,c (a   +  2r) + b+\delta)$ because   
$\varrho(y - w,z) = \varrho(y, v) = ct $
and
$$
\varrho(y-w,\bz)\leq \varrho(y-w,z)+\varrho(z,\bz) < ct +  c (a   +  r) + b +\delta \leq c (a   +  2r) + b+\delta.
$$
By (A), there is a $u^{\prime }\in \ball_X[u,t]$ and $z'\in F(u')$ such that $\varrho(y - w,z')<(c-c')t/2$. So 
$d(u^{\prime }, \bx)  \leq d(u^{\prime },u)  + d(u, \bx)  < t+  a+r \leq a+2r$.
As $w\in H(u) \cap \ball_Y(\bar{w},(a + r) \ell+\delta)$, using {\rm (B)}, we find a $w'\in H(u^{\prime })$ such that 
$\varrho (w,  w')  \leq \ell \,  d(u,u') \leq \ell t$. 
Let $v^{\prime }:=z' +w' \in  G(u')$.
 Then
\begin{eqnarray}
	\varrho(v', y)  & = &  \varrho(z' + w', y)\leq \varrho(z' + w',z'+w) +\varrho(z'+w,y) \notag \\
	&=&   \varrho(w,w')+ \varrho(y-w, z') < \ell t+ (c-c')t/2 
	  =   (c+c')t/2  - (c' - \ell) t  \label{eqB001} \\
	  &  = & \lambda \,   \varrho (v, y)- (c' - \ell) t. \label{eqB003}
\end{eqnarray}
Inequality in \eqref{eqB001} reveals  that 
\begin{eqnarray*}
	\varrho(v,v')  & \leq & \varrho(v,y) +\varrho(y, v') <  ct+ \ell t+ (c-c')t/2 =  (3c/2 - c'/2 + \ell) t \\
	&<& (3c/2 + c'/2) t  < 2ct = t/\alpha.
\end{eqnarray*}
Hence $\alpha  \varrho(v, v') < t$; and as $d(u,u') \leq t$,  \eqref{eqB003} implies \eqref{eqBetterPerturb}.

Proposition~\ref{propAlmostRegSetvalued2}, with $c := c' - \ell$, $U:=\ball_{{X}}(\bx,a)$, $V:= \ball_{Y}(\bz+\bw, b)$, and $\gamma \equiv r$,  implies the conclusion, cf. Remark~\ref{rmkGamma} and Remark~\ref{rmkClosedBalls}.
\end{proof}

The above statement covers approximate versions of several results. First, we consider global Lyusternik-Graves theorem, see \cite[Theorem 5I.2]{DontchevBook}, \cite[Corollaries 3.2 and 3.11]{CR2023}, or \cite[Corollary 3.4]{HT08}: 
\begin{corollary} \label{corPerturbGlobal}
	Let $(X,d)$ be a  metric space and $(Y, \varrho)$ be a linear metric space with a~translation-invariant metric. Consider $c$, $\ell \in (0, \infty)$ such that  $\ell <  c$, an $r \in (0, \infty]$, and set-valued mappings $F$, $H:X\rightrightarrows Y$. Assume that $F$ has property $(\mathcal{O})$ on $X \times Y$ for $c$ and $\gamma \equiv r$ and that $H$ is Hausdorff-Lipschitz on $X$ with the constant $\ell$. Then the mapping  $F+H$ has property  $(\mathcal{O})$ on $X \times Y$ for $c - \ell$ and $\gamma \equiv r$.
\end{corollary}
\begin{proof}  Fix any $(\bx,\bv)\in \gph\,(F+H)$ and any $a \in (0, \infty)$. Let $b : = \infty$ and $\delta :=\diam\, H(\bx)$.  Find $\bz\in F(\bx)$ and $\bw\in H(\bx)$ such that $\bv=\bz+\bw$. Clearly, assumption (B) holds true.  Given arbitrary $u \in X$,  $z \in F(u)$, and $t \in (0, r)$, we have  $\ball_Y(z, c t) \subset  \overline{F(\ball_X(u,t))} \subset \overline{F(\ball_X[u,t])}$; hence $\ball_Y[z, c t] \subset  \overline{F(\ball_X[u,t])}$. Therefore (A) holds true. To show (C), fix any $u\in \ball_{{X}}(\bx, a + r)$ and $v\in (F+H)(u)$. There is a $w\in H(u)$ such that $v\in F(u)+w$.  Find $\bar{w}'\in H(\bx)$ such that $\varrho(w,\bar{w}')\leq \ell d(u,\bx)<\ell (a + r)$. 
	Then
	\begin{equation} \label{eqEstDiam}
	\varrho(w,\bar{w})\leq	\varrho(w,\bar{w}')+	\varrho(\bar{w}',\bar{w})< \ell (a + r) +  \delta;
	\end{equation}
	therefore {\rm (C)} holds.
Theorem~\ref{thmApproxL-Gsetset} implies that $F+H$ has property  $(\mathcal{O})$ on $ \ball_{{X}}(\bx, a) \times Y$ for $c - \ell$ and $\gamma \equiv r$; where $a > 0$  is arbitrary.
\end{proof}

\begin{remark} \label{rmkGlobal} \rm  Let the assumptions of Corollary~\ref{corPerturbGlobal}, with $r = \infty$, be satisfied. Let $\kappa := 1/c$. Then, in view of Proposition~\ref{propApproxRegularitiesEquiv} and Remark~\ref{rmAlmostRegularSeq},   
 for each $x \in X$, each $y \in Y$, and each $\varepsilon > 0$ there is a~sequence $(y_k)_{k \in \mathbb{N}}$ converging to $y$ such that 
$$
  	\dist\big(x,(F+H)^{-1}(y_k)\big)\leq \frac{\kappa}{1-\kappa \ell}  \, \big(\dist(y, (F+H)(x)) + \varepsilon\big) 
  	\quad \mbox{for each} \quad k \in \mathbb{N}.
$$
Hence we get the conclusion in \cite[Theorem 13]{ABB2020}, where $F$ is assumed to be globally regular (instead of \emph{almost} globally regular).
\end{remark}
 
Second, we derive approximate analogues of \cite[Corollaries 3.4 and 3.13]{CR2023} going back to \cite[Theorem 3.2]{ACN15}:

\begin{corollary}
	Let $(X,d)$ be a metric space and $(Y,\varrho)$ be a linear metric space with a translation-invariant metric. Consider positive $a$, $b$, $c$, and $\ell$ such that  $ \ell < c$,  a point $(\bar{x}, \bar{z}, \bar{w}) \in X \times Y \times Y$, and set-valued mappings $F$, $H:X\tto Y$ with $\bar{z}\in F(\bar{x})$ and $\bar{w} \in H(\bar{x})$. Assume that 
	$F$  has property $(\mathcal{O})$ on  $\ball_{{X}}(\bx,a)\times\ball_{Y}(\bz, b)$ for $c$ and $\gamma \equiv \infty$, that  $H$ is Hausdorff-Lipschitz on $\ball_{{X}}(\bar{x}, a)$ with the constant $\ell$, and that $\diam\,H(\bx)< b$. Then for each $\beta  > 0$ such that
	\begin{eqnarray}
		\label{eqSetConstants}
\quad  \beta (3+2/(c-\ell)) <a 
   \ \mbox{and} \  \beta(3c + 2c/(c - \ell)+1) +  \diam\,H(\bx)< b,
	\end{eqnarray}
	 the mapping $F+H$ has property  $(\mathcal{O})$ on $ \ball_{{X}}(\bx, \beta) \times \ball_{Y}(\bz+\bw, \beta)$ for $c - \ell$ and $\gamma \equiv \infty$.
\end{corollary}
\begin{proof} Let $\beta$ satisfy \eqref{eqSetConstants}. Put $\delta:=\diam\,H(\bx)$ and $r := (1+c-\ell) \beta/(c-\ell)$.  We check  assumptions (A), (B), and (C) of Theorem~\ref{thmApproxL-Gsetset} with $(a,b) := (\beta, \beta)$. Assumption (B) holds true because $\beta + 2r = \beta (3 + 2/(c-\ell)) < a$. Given arbitrary $u \in\ball_{{X}}(\bx, a)$,  $z \in F(u)$, and $t\in (0, r)$, we have 
$\ball_Y(z, c t) \cap\ball_Y(\bz, b) \subset \overline{F(\ball_X(u,t))} \subset \overline{F(\ball_X[u,t])}$; thus $\ball_Y[z, c t] \cap\ball_Y(\bz, b) \subset \overline{F(\ball_X[u,t])}$. Since  $\beta + r = \beta(2+1/(c - \ell)) < a$ and  $c (\beta   +  2r) + \beta + \delta = \beta(3c + 2c/(c - \ell)+1) + \delta < b$, we conclude that (A) is satisfied. To show that (C) holds, fix any $u\in \ball_{{X}}(\bx, \beta +r)$ and $v\in (F+H)(u)$. There is a $w\in H(u)$ such that $v\in F(u)+w$. As $H$ is Hausdorff-Lipschitz on $\ball_{{X}}(\bar{x}, a)$ and $\beta + r  < a$,    there is a $\bar{w}'\in H(\bx)$ such that $\varrho(w,\bar{w}')\leq \ell d(u,\bx)<\ell (\beta + r)$.
	Hence \eqref{eqEstDiam} with $a := \beta$ holds; therefore so does (C). Theorem~\ref{thmApproxL-Gsetset} implies that $F+H$ has property  $(\mathcal{O})$ on $ \ball_{{X}}(\bx, \beta) \times \ball_{Y}(\bz+\bw, \beta)$ for $c - \ell$ and $\gamma \equiv r = (1+c-\ell) \beta/(c - \ell)$.   Noting that $\min\{\beta, \beta, r(c-\ell)/(1+c - \ell)\} = \beta$, Remark~\ref{rmkMetricRegularity} and Proposition~\ref{propApproxRegularitiesEquiv} reveal that $F+H$ has property $(\mathcal{O})$ on $\ball_X(\bx,\beta) \times \ball_Y(\bar{y},\beta)$ for $c - \ell$ and $\gamma \equiv \infty$.
\end{proof}

Instead of the condition on the diameter  one can assume the so-called local sum-stability introduced by M. Durea and R. Strugariu \cite{DS2012}.

\begin{definition} \label{lss} Let $(X,d)$ be a metric space and $(Y,\varrho)$ be a linear metric space. Consider a point $(\bx,\bz, \bw)\in X\times Y\times Y$ and set-valued mappings  $F$, $H: X\rightrightarrows Y$ such that $\bz\in F(\bx)$ and  ${\bw\in H(\bx).}$ We say that the pair \emph{$(F,H)$ is sum-stable around $(\bx,\bz,\bw)$} if
	for each $\xi >0$ there exists a $\beta >0$ such that, for each $u\in \ball_X(\bx, \beta)$ and each $v\in
	(F+H)(u)\cap \ball_Y({\bz+}\bw, \beta)$, there exist $z\in F(u)\cap \ball_Y(\bz, \xi)$ and $w\in H(u)\cap \ball_Y(\bw,\xi)$ such that $v=z+w.$
\end{definition}

Using the above notion, we  obtain an approximate version of \cite[Corollaries 3.7 and 3.15]{CR2023} going back to \cite[Theorem 8]{HNT2014} and generalizing Theorem~\ref{thmApproxL-Gset}, because if a mapping $h: X \to Y$ is single-valued and Lipschitz continuous around (even calm at)  $\bx$, then the pair $(F,h)$ is sum-stable around $(\bx,\bz, h(\bx))$ for any   $F: X\rightrightarrows Y$ with  $\bz\in F(\bx)$. 

\begin{corollary} \label{thmPeturbSetvalued}
	Let $(X,d)$ be a  metric space and $(Y, \varrho)$ be a linear metric space with a~translation-invariant metric.  Consider  a point  $(\bx,\bz, \bw)\in X\times Y\times Y$ and set-valued mappings $F$, $H:X\rightrightarrows Y$ with $\bz\in F(\bx)$ and $\bw\in H(\bx)$. 
	Assume that $\lip H (\bx, \bw) < \infty$ and that the pair $(F,H)$ is sum-stable around $(\bx,\bz, \bw)$. 
	Then  
	\begin{equation} \label{eqApproxSurfFsetSet}
		\overline{\sur}(F+H)(\bx, \bz + \bw) \geq  \overline{\sur}\,F(\bx, \bz) - \lip H (\bx, \bw).
	\end{equation}
\end{corollary}
\begin{proof} Assume that $\overline{\sur}\,F(\bx, \bz) > \lip H (\bx, \bw)$ because otherwise there is nothing to prove.  Fix constants  $c$  and $\ell$ such that 
	$\lip H (\bx, \bw) < \ell < c < \overline{\sur}\,F(\bx, \bz)$.
 Find a $\beta>0$ such that $F$ has property $(\mathcal{O})$ on  $\ball_{X}(\bx, \beta) \times \ball_{Y}(\bz, \beta)$ for $c$ and $\gamma \equiv \beta$ and that $H$ has Aubin property on $\ball_{X}(\bx, \beta )\times \ball_{{Y}}(\bw, \beta)$ with the constant $\ell$. Let $\delta := \beta/2$. Find a $\beta' \in (0, \beta/2)$ such that  for each $u\in \ball_X(\bx,  \beta' )$ and each $v\in 	(F+H)(u)\cap \ball_Y({\bz+}\bw, \beta')$, there exist $z\in F(u)\cap \ball_Y(\bz, \delta)$ and $w\in H(u)\cap \ball_Y(\bw,\delta)$ such that $v=z+w$. Let $r :=  \beta'/(3c +2)$.

We check  assumptions (A), (B), and (C) of Theorem~\ref{thmApproxL-Gsetset} with $(a,b) := (r, r)$.  Assumption (C) holds true because $r + r < \beta'$, $r + cr < \beta'$,  and  $(r + r)\ell + \delta > \delta$.  Since  $(r + r)\ell + \delta < 2rc + \delta <  \beta' + \delta < \beta$ and $r + 2r <3\beta'/2 < \beta$, assumption (B) is verified.   Given arbitrary $u \in\ball_{{X}}(\bx, \beta)$,  $z \in F(u)$, and $t\in (0, \beta)$, we have 
$\ball_Y(z, c t) \cap\ball_Y(\bz, \beta) \subset \overline{F(\ball_X(u,t))} \subset \overline{F(\ball_X[u,t])}$; thus $\ball_Y[z, c t] \cap\ball_Y(\bz, \beta) \subset \overline{F(\ball_X[u,t])}$.  Since  $r + r <  \beta' < \beta $ and  $c (r   +  2r) + r + \delta  < \beta' + \delta < \beta$, we conclude that (A) is satisfied.

Theorem~\ref{thmApproxL-Gsetset} implies that $\overline{\sur} \, (F+H) (\bx, \bz + \bw) \geq c- \ell$. Letting $c\uparrow \overline{\sur} \, F(\bx, \bz)$ and $\ell \downarrow \lip H(\bx, \bw)$ we finish the proof.
\end{proof}

\section{Linear and Affine Mappings} \label{scLinearMappings}

A topological vector space $X$, that is, a vector space over $\mathbb{R}$ equipped with a~Hausdorff topology such that the mappings $(u,x)\longmapsto u+x$ and $(\lambda,x) \longmapsto  \lambda x$
are continuous on the spaces $X\times X$ and $\mathbb{R}\times X$ with respect to their product topologies, is said to be a
\emph{locally convex space} if every neighborhood of the origin  contains a convex neighborhood of the origin. Every neighborhood of the origin is an absorbing set. A locally convex space $X$ is said to be \emph{barreled} if each closed convex circled absorbing subset of $X$ is a neighborhood of the origin in $X$.  In every locally convex space there is a neighborhood basis of the origin
consisting of closed convex circled (hence symmetric) sets. Every Fr\'echet space, i.e., a complete metrizable locally convex space,  and every Baire space, i.e. countable unions of closed sets with empty interior also have empty interior, is barreled. 

A linear operator $A$ acting between topological vector spaces $X$ and $Y$ is said to be \emph{almost open at $x \in X$} if for each neighborhood $U$ of $x$ in $X$ the set $\overline{A(U)}$ is a~neighborhood of $Ax$ in $Y$. In metrizable topological vector spaces,  $A$ is almost open around each point if and only if $A$ is almost open at each point if and only if $A$ is almost
open at the origin. Moreover, if $A$ acts between normed spaces then 
$$
 \overline{\sur} A = \sup\big\{c > 0: \ \overline{A(\ball_X)} \supset c \ball_Y \big\}. 
$$
The almost openness modulus is a Lipschitz continuous function on the space of bounded linear operators $\mathcal{BL}(X,Y)$ equipped with the operator norm; this is a~quantitative version of \cite[Theorem 1.2]{Harte84}. 
\begin{corollary} \label{corApproxLinear}
Let $(X, \|\cdot\|)$ and $(Y, \|\cdot\|)$ be normed spaces. The function $\mathcal{BL}(X,Y) \ni A \longmapsto \overline{\sur} A$ is Lipschitz continuous with the constant $1$. In particular, the almost open linear operators form an open subset of $\mathcal{BL}(X,Y)$. 
\end{corollary}  

\begin{proof}
 Let $A$, $B \in \mathcal{BL}(X,Y)$ be arbitrary. As $\lip(A-B) = \|A - B\|$,  Theorem~\ref{thmApproxL-Gset} implies  that 
$$
 \overline{\sur} B \geq \overline{\sur} A -  \|A-B\|
 \quad \mbox{and} \quad 
 \overline{\sur} A \geq \overline{\sur} B -  \|A-B\|.
$$
Since  $\overline{\sur} A \leq  \|A\| < \infty$ and $\overline{\sur} B \leq  \|B\| < \infty$, a combination of the displayed formulas above yields that 
$$
  -  \|A-B\| \leq \overline{\sur} A  - \overline{\sur} B \leq \overline{\sur} A - \overline{\sur} A  + \|A - B\| = \|A-B\|.
$$
In particular, given an $A \in  \mathcal{BL}(X,Y)$ with $\overline{\sur} A > 0$, we have $\overline{\sur} B > 0$ for each $B \in \mathcal{BL}(X,Y)$ with $\|A - B\| < \overline{\sur} A$. 
\end{proof} 

It is elementary that if an $A \in \mathcal{BL}(X,Y)$ is almost open then  $\overline{A(X)} = Y$.  On the other hand, a sufficient condition for almost openness of a linear bounded mapping is contained in \cite[Theorem 1.3]{Harte84}. 

\begin{proposition} \label{propHarte}
Let $(X, \|\cdot\|)$ and $(Y, \|\cdot\|)$ be normed spaces. Consider  an $A \in \mathcal{BL}(X,Y)$ such that  $\overline{A(X)} = Y$ and that there is an $\alpha > 0$ such that $ \alpha \|x\| \leq \|Ax \|$ for each $x \in X$. Then $\overline{\sur} A \geq \alpha$. 
\end{proposition}  

\begin{proof}
  We show that $\overline{A(\ball_X)} \supset \alpha \ball_Y$. Let a non-zero $y \in \alpha \ball_Y$ be arbitrary. There is a sequence $(x_k)$ in $X$ such that $\|y - Ax_k\| \to 0$ as $k \to \infty$. Assume without any loss of generality that  $0 < \|Ax_k\| < 2 \alpha$ for each $k \in \mathbb{N}$.  Let 
  $ u_k := \hbox{$\frac{\|y\|}{\|Ax_k\|}$} x_k$, $k \in \mathbb{N}$. For each $k \in \mathbb{N}$, we have $\alpha \|u_k \| \leq \|Au_k\| = \hbox{$\frac{\|y\|}{\|Ax_k\|}$} \|Ax_k\| = \|y\| \leq \alpha$, thus $u_k \in \ball_X$; and   $\|y - Au_k\| \leq \|y - Ax_k\| + \|Ax_k - Au_k \| = \|y - Ax_k\| + \big|\|Ax_k\| - \|y\| \big| \ \leq 2 \|y - Ax_k\| \to 0$ as $k \to \infty$.
\end{proof}

Any $A \in \mathcal{BL}(X,Y)$ satisfying the latter assumption in Proposition~\ref{propHarte} is said to be \emph{bounded below} in \cite{Harte84} and by \cite[Theorem 1.4]{Harte84}: {\it If  $A \in \mathcal{BL}(X,Y)$ is in the topological boundary of the set of all linear almost open mappings then $A$ is not bounded below.}

The common proof of the classical Banach's open mapping theorem, e.g. in Fr\'echet spaces, is divided into two steps. First, the almost openness of the linear surjective
mapping under consideration is established. Second, one shows that the almost openness implies the openness provided that this  mapping is continuous.  
A variant of the open mapping theorem for generally non-metrizable locally convex spaces goes back to Pt\'{a}k \cite{Ptak58}. Let us focus on the first step and formulate a~simple statement:

\begin{proposition} \label{propPtak}
Let $X$ and $Y$ be locally convex spaces and $Y$ be barreled.  Then each linear mapping from $X$ {\tt onto} $Y$ is almost open.
\end{proposition}

\begin{proof}
Pick an arbitrary linear mapping $A: X \to Y$ such that $A(X)=Y$. Let $U$ be a neighborhood of $0$ in $X$. There exists a convex circled neighborhood $U'$
of $0$ such that $U'\subset U$. Then $V:=\overline{A(U')}$ is a closed convex
circled set. Being a neighborhood of $0$, the set $U'$ is absorbing in $X$, hence $A(U')$ is absorbing in $A(X)=Y$; and therefore so is $V$. Since $Y$ is barreled, the set $V$ is
a neighborhood of $0$ in $Y$, hence so is $\overline{A(U)}$. 
\end{proof}

Now, we focus on affine mappings. 

\begin{definition}\label{defAffMap}
  Let $X$ and $Y$ be vector spaces.  A mapping $A : X \to Y$ is said to be \emph{affine} 
  if the domain of $A$ is convex and  $ A(\lambda u+(1-\lambda)x)=\lambda A(u)+(1-\lambda)A(x)$ for each $u$, $x\in \dom A$  and each $\lambda\in[0,1]$.
\end{definition}

\noindent A restriction of a linear mapping on a convex subset $K$ of $X$ is affine. On the other hand, any affine mapping $A$ can be extended to an affine mapping $\widetilde{A}$ defined on the whole space $X$ such the restriction of $\widetilde{A}$ to $\dom A$ is $A$ and the mapping $\widetilde{A} - \widetilde{A}(0)$  is linear.
The next result generalizes the first step of the proof of the Banach open
mapping theorem by following a standard pattern which combines the Baire's category theorem \cite[Theorem 5.1.1]{Jarchow} and the \emph{line segment principle} \cite[Proposition 6.2.2]{Jarchow}:  \emph{Let $K$ be an open convex subset of a~Hausdorff topological vector space and let $\lambda \in (0,1]$. Then $\lambda \, K + (1-\lambda) \overline{K} \subset K$.}

\begin{proposition} \label{propAlmostOpen}
  Let $X$ be a locally convex space and $Y$ be a complete metrizable topological vector space.  
  Consider a point $x \in X$, an affine mapping $A: X \to Y$ with   $x \in \dom A$  and $ A(x) \in \core A(X)$, and 
  a sequence $(K_p)$ of  convex subsets of $\dom A$ such that 
  \begin{equation} \label{eqKnK}
    K_\infty := \dom A = \bigcup_{p \in \mathbb{N}} \,  K_p
    \quad \mbox{and} \quad x \in K_p \subset K_{p+1} 
    \quad \mbox{for each} \quad p \in \mathbb{N}.
  \end{equation} 
  Let $\alpha \in (0,1]$ and let $U$ be a neighborhood of $x$ in $x + \alpha(K_\infty-x)$. Then there exists a $q \in \mathbb{N}$ such that  the set 
  $\overline{A\big(U\cap (x+ \alpha (K_{p} - x))\big)}$ is a~neighborhood of $A(x)$ in $Y$ for each $p  \in \{q, q+1, \dots\}$.  If, in addition, for such an index $p$  
  the set $U\cap K_{p}$ 
  is bounded, then for each neighborhood $W$ of $x$ in $x+ \alpha (K_{p} - x)$ the set $\overline{A(W)}$ is a~neighborhood of $A(x)$ in $Y$.\end{proposition}

\begin{proof}
   Without any loss of generality assume that $x=0$ and $A(0)=0$. Let $U$ and $\alpha$ be as in the premise. Let $K:= \alpha K_\infty$. Since $0 \in \core A(K_\infty)$ and $A(K) = \alpha A(K_\infty)$, we have $0 \in \core A(K)$.   Find a~convex neighborhood $V$ of $0$ in $X$ such that  $ V \cap K \subset U$. Let $V_p:= V \cap \alpha K_p$, $p \in \mathbb{N}$.
  
  Fix an arbitrary $y\in Y$. As $0\in \mbox{\rm core} \, A(K)$, we have  $\bigcup_{\lambda > 0} \lambda A(K) = Y$.  Find  an $i \in \mathbb{N}$ such that
  $y = iA(x) $ for some $x \in K$.  Since $V$ is a convex neighborhood of $0$ in $X$ and $K$ is the union of convex sets $\alpha K_p$ (all containing $0$), there exist positive integers
  $j$ and $l$ such that $x/j \in V\cap \alpha K_{l} =:V_{l}$. Thus
  $$
    y= iA(x)=ij\left(\hbox{$\frac{1}{j}$} A(x) + \big(1-\hbox{$\frac{1}{j}$}\big) A(0) \right) =ijA\big(\hbox{$\frac{1}{j}$}x\big)\in  ij A(V_{l}).
  $$
  Since $y\in Y$ is arbitrary, we get that 
  \begin{equation} \label{eqUnion}
     \bigcup_{k=1}^\infty\bigcup_{p=1}^\infty k\,A(V_p)=Y, 
     \quad \mbox{where} \quad 
      V_{p}:= V\cap \alpha K_{p}, \ p \in \mathbb{N}. 
  \end{equation}
  Baire's category theorem says that there is a $q \in \mathbb{N}$ such that the set $M_{q}:= \overline{A(V_{q})}$ has an interior point, say $y$. 
  As $V_p\subset V_{p+1}$ for each $p \in \mathbb{N}$, \eqref{eqUnion} reveals that, taking larger $q$ if necessary, 
  we may   assume 
  that $-y \in k M_{q}$ for some $k \in \mathbb{N}$.  The set $V_{q}$ is convex, thus so are $A(V_{q})$ and $M_{q}$.  
  Since $-y/k \in M_{q}$ and $y \in \intt M_q$, 
  the line segment principle says that 
  $$
    0 \in \intt M_q \subset M_{q} = \overline{A(V\cap \alpha K_{q})} 
    =  \overline{A(V\cap K \cap \alpha K_{q})} \subset \overline{A(U\cap \alpha K_{q})}.
  $$
  As $\alpha K_q\subset  \alpha  K_{q+1} \subset \cdots  \subset \alpha K_\infty$, the proof of the first part is finished.
  
  Fix any $p \in \{q,q+1, \dots\}$  such that $U_{p}:= U\cap K_{p}$ is a bounded set. Let $W$ be an arbitrary  neighborhood of $0$ in $\alpha K_{p}$. 
  Find a convex neighborhood $\Omega$ of $0$ in $X$ such that  $\Omega \cap \alpha K_{p} \subset W$. Then there exists a $\lambda>0$ such that
  $U_{p} \subset \lambda \Omega$. As $\Omega$ is convex, we may assume that
  $ \lambda \geq 1$.  Then $\alpha K_p  \subset \lambda \alpha K_{p}$ because $0 \in K_{p}$, thus
  $$
   U \cap \alpha K_p =  U \cap K_p \cap \alpha  K_p  = U_{p} \cap \alpha K_p \subset \lambda \Omega \cap \lambda \alpha K_{p} =\lambda (\Omega\cap \alpha K_{p}).
  $$
  For $\gamma:= 1/\lambda$, the set  $\gamma M_{q}$ is a neighborhood of $0$ in $Y$ and  
  $$
     \gamma M_{q} \subset \gamma  \overline{A(U\cap \alpha K_{p})} \subset \overline{A\left(\gamma (U \cap \alpha K_p)  \right)} \subset \overline{A(\Omega\cap \alpha K_{p})} \subset \overline{A(W)} .
  $$

\end{proof}

\section{Concluding remarks and future work}
\paragraph*{\bf R$_1$}   Assumption  $(\mathcal{A}_3)$ seems to be crucial in order to get full forms of Theorems~\ref{thmEkelandWeak} and \ref{thmEkelandWeakE} as well as of the corresponding corollaries described in Remark~\ref{rmkComentarySC2}. This consideration is left for an interested reader (if any).

\smallskip 
\paragraph*{\bf R$_2$} Definition~\ref{defApproxRegularities} and Proposition~\ref{propApproxRegularitiesEquiv} can be generalized in three basic ways. First, we can allow $\mu$ to be a continuous non-negative strictly increasing
function on $[0,\infty)$ such that $\lim_{t \downarrow 0} \mu(t) = \mu(0) = 0$ and $\lim_{t \to \infty} \mu(t) = \infty$ instead of $\mu(t) := \mu t$. Second, one can consider a non-empty set $W \subset X \times Y$ instead of $U \times V$. These modifications are obvious, cf.  \cite[Definition~2.92 and Theorem~2.95]{IoffeBook} and \cite[Proposition 2.8]{NDH2022}, respectively. Roughly speaking, one adds brackets around the quantities following $\mu$ and works with $\mu^{-1}$ instead of $1/c$ in the first case. In the latter, one replaces  $U$ and $V$ by either the projections of $W$ onto $X$ and $Y$, respectively, defined by
$W_X := \{ x \in X: \ (x,y) \in W \mbox{ for some } y \in Y \}$  and $W_Y := \{y \in Y: \ (x,y) \in W \mbox{ for some } x \in X \}$; or the fibers $W_{\cdot,y} := \{x \in X: \ (x,y) \in W \}$ for $y \in W_Y$ and $W_{x,\cdot} := \{y \in Y: \ (x,y) \in W \}$ for $x \in W_X$, see also Proposition~\ref{propAlmostRegW} below. Third, one can consider directional versions of the properties as in \cite{CR2020, CDPS}. In the case of openness, one uses ``balls'' $\ball^\eta_X(x,r)$, where $\eta$ satisfies $(\mathcal{A}_1)-(\mathcal{A}_3)$, instead of the usual ones; while for the remaining two equivalent properties one employs the so-called minimal time function with respect to a set, see \cite[Definition 1 and Proposition 3]{CDPS}.

\smallskip 
\paragraph*{\bf R$_3$} Similarly to Corollaries~\ref{corAlmostFabianPreiss} and \ref{corAlmostRegSetvalued2}, taking $U\times V$ as $U \times \{\bar{y}\}$ and $\{\bar{x} \} \times V$, respectively, we obtain (obvious) sufficient conditions for almost subregularity and semiregularity at a point. More importantly, we can replace $U \times V$ by a general non-empty set $W \subset X \times Y$ to obtain approximate versions of statements in \cite{CR2023}, e.g. Proposition 2.6 therein. Moreover, using Theorem~\ref{thmApproxSequentialEkeland}, one can merge regularity and almost regularity together, e.g. Proposition~\ref{propFabianPreiss} and Corollary~\ref{corAlmostFabianPreiss}. Such a generalized Proposition~\ref{propAlmostRegSubreg} reads as:
\begin{proposition}\label{propAlmostRegW} Let $(X,d)$ and $(Y,\varrho)$ be metric spaces. Consider a $c> 0$, a~non-empty set $W \subset X \times  Y$, a function $\gamma: X \to [0, \infty]$  being  defined and not identically zero on $W_X$, and a mapping  $g: X \to Y$ defined on the whole $X$.  Assume that for each $y \in W_Y$  there is a $\lambda = \lambda(y) \in (0,1)$ such that for each $u \in X$ satisfying
%$$
%\varepsilon < \varrho(g(u),y) \leq \varrho(g(x),y) -  c \, d(u,x)
% \mbox{ and } 
% \varrho(g(x),y) <   c \gamma(x)
% \mbox{ for some }  x \in W_{\cdot,y},
%$$
\eqref{eqAlmostConstraint} for some $x \in W_{\cdot,y}$, 
there is a point  $u'\in X$ %, better than $u$, 
such that \eqref{eqAlmostBetter} holds true.  Then for each $x \in W_X$,  each $t  \in (0,\gamma(x))$, and each $y \in \ball_Y(g(x),ct) \cap  W_{x,\cdot}$ there is a Cauchy sequence $(u_k)_{k \in \mathbb{N}}$ in  $\ball_X\big[x,\varrho(g(x),y)/c\big]$ such that $\lim_{k  \to \infty} \varrho(g(u_k), y) = 0$. 

In particular, $\ball_Y(g(x),ct) \cap  W_{x,\cdot} \subset \overline{g\big(\ball_X(x,t)\big)} $ for each $x \in W_X$ and each $t  \in (0,\gamma(x))$. If, in addition, $X$ is complete and $g$ is continuous, then $\ball_Y(g(x),ct) \cap  W_{x,\cdot} \subset g\big(\ball_X(x,t)\big)$ for each $x \in W_X$ and each $t  \in (0,\gamma(x))$.
\end{proposition}

\begin{proof}
Let $x \in W_X$,  $t  \in (0,\gamma(x))$, and  $y \in \ball_Y(g(x),ct) \cap  W_{x,\cdot} \subset W_Y$  be arbitrary.  

If $y = g(x)$, then put $u_k:=x$, $k \in \mathbb{N}$, and we are done. Assume that this is not the case.  Let $\varphi:=\varrho (g(\cdot), y)$ and $\eta := c d$. Apply Theorem~\ref{thmApproxSequentialEkeland}, to find an $N \in \mathbb{N} \cup \{\infty\}$ and $(u_k)_{k = 1}^{N}$, with $u_1=x$, satisfying (i) and (ii) therein. 
If $N < \infty$, then put $u_k:= u_N$, $k \in \{N+1, N+2, \dots\}$. Let $\beta_k := \varrho(g(u_k),y)$, $k \in \mathbb{N}$. By (i), we get that 
\begin{equation} \label{eqCauchy}
 (0 \leq) \quad   c\, d(u_j,u_k) \leq \beta_k - \beta_j 
 \quad \mbox{for each} \quad j, k \in \mathbb{N} \quad \mbox{with} \quad k \leq j. 
\end{equation}
As $u_1 =x$,  the sequence $(u_k)_{k \in \mathbb{N}}$ lies in  $\ball_X\big[x,\varrho(g(x),y)/c\big]$ because $\beta_1 = \varrho(g(x),y)$. Also the sequence $(\beta_k)_{k \in \mathbb{N}}$ is decreasing (and bounded from below); hence it converges, to a $\beta \in \big[0,\varrho(g(x),y)\big]$. Employing \eqref{eqCauchy} again, we infer that $(u_k)_{k \in \mathbb{N}}$ is a Cauchy sequence because $(\beta_k)_{k \in \mathbb{N}}$ is such.  Assume that $\beta > 0$.  Find a $\lambda \in (0,1)$ such that the assumption containing \eqref{eqAlmostConstraint}  and \eqref{eqAlmostBetter} holds true. By (ii) in Theorem~\ref{thmApproxSequentialEkeland} with $\varepsilon := (1-\lambda) \beta$, there is an index $j \in \mathbb{N}$ such that 
\begin{equation} \label{eqApproxEVP02uB}
	  \varrho(g(u'),y) +  c \, d(u',u_j) > \varrho(g(u_j),y) -  (1-\lambda) \beta \quad \mbox{for each} \quad u' \in X. 
\end{equation}
Since $(\beta_k)_{k \in \mathbb{N}}$ is decreasing,  \eqref{eqCauchy}, with $k = 1$, implies that  
$$
  \beta \leq \beta_j =  \varrho(g(u_j),y) \leq  \varrho(g(x),y)  -  c \, d(u_j,x).  
$$
As $\varrho(g(x),y) <  ct < c \gamma(x)$ and $x \in W_{\cdot,y}$ because $(x, y) \in W$, there is a  $u' \in X$ such that 
\begin{equation} \label{eqAlmostBetterK}
  \varrho(g(u'),y) + c\,d(u_j,u') \leq \lambda \, \varrho(g(u_j),y).
\end{equation}
Since  $\beta \leq  \varrho(g(u_j),y)$, we get that  	
    \begin{eqnarray*}
       c\,d(u_j,u') & \overset{\eqref{eqAlmostBetterK}}{\leq} &  \lambda \varrho(g(u_j),y) - \varrho(g(u'),y)  \leq   \varrho(g(u_j),y) - (1-\lambda) \beta - \varrho(g(u'),y) \\
       &  \overset{\eqref{eqApproxEVP02uB}}{<} &  
           \varrho(g(u'),y) +  c \, d(u',u_j)  - \varrho(g(u'),y)   =  c\,d(u',u_j), 
     \end{eqnarray*}
  a~contradiction. Hence $\beta =0$, which proves the first conclusion.  As $\varrho(g(x),y) < ct$,  the (whole) sequence $(u_k)_{k \in \mathbb{N}}$ lies in  $\ball_X(x,t)$. Consequently, $y \in \overline{g\big(\ball_X(x,t)\big)}$.   
  
  Assume, in addition, that $X$ is complete and $g$ is continuous. Then  $(u_k)_{k \in \mathbb{N}}$ converges, to a  $u \in \ball_X\big[x,\varrho(g(x),y)/c\big] \subset \ball_X(x,t)$ and 
  $ 0 =  \lim_{k  \to \infty} \varrho(g(u_k), y)  =  \varrho(g(u),y)$. Consequently, $y \in g\big(\ball_X(x,t)\big)$.   
 \end{proof}

Obviously, one can also obtain versions of  \cite[Proposition 2.7 and Corollary 2.8]{CR2023}. Finally, as described in {\bf R$_2$} one can consider directional versions of the properties to get approximate versions of the Ioffe-type criteria in \cite{CR2020, CDPS}.
  
\smallskip 
\paragraph*{\bf R$_4$} Employing \emph{equivalent} set-valued version of Proposition~\ref{propAlmostRegW}, Theorem~\ref{thmApproxL-Gset} can be generalized to cover \cite[Theorem 3.8]{CR2023}, implying \cite[Theorems 4.2 and 4.3]{NDH2022}, as well as an approximate version of it in incomplete spaces for regularity on a fixed set $W \subset X \times Y$. When the paper was ready for submission we learned about \cite[Theorems 5.1(a) and 5.2(a)]{N2023}, being precisely statements of this type for Milyutin regularity. By our approach, we get almost regularity of $F + h$ under the (weaker) assumption that $F$ is Milyutin \emph{almost} regular on $U \times V$ instead of Milyutin  regular on this set. 

Second, assumption (B) in Theorem~\ref{thmApproxL-Gsetset} can be augmented as 
\begin{itemize}
  		\item[$(\widetilde{B})$]  $H(u) \cap \ball_Y(\bar{w}, (a + r) \ell+\delta) \subset \overline{H(u') + \ball_Y[0, \ell \, d(u,u')]}$  for each $u$, $u' \in \ball_{{X}}(\bar{x}, a + 2r)$.
\end{itemize}
Indeed, in the proof, one finds $z' \in F(u')$ such that $\varrho(y-w,z') < (c - c')t/4$  and then, using   $(\widetilde{B})$, a $w' \in H(u')$ such that $\varrho(w,w') \leq \ell \,  d(u, u') + (c - c')t/4$. 

Finally, one can investigate directional versions as well as the almost regularity of compositions similarly to \cite{CDPS}.   

\medskip
\paragraph*{\bf Acknowledgement} The author is grateful to Abderrahim Hantoute, University of Alicante, and to Tom\'a\v{s} Roubal, The Institute of Information Theory and Automation (UTIA) in Prague, for a detailed discussion on numerous preliminary versions of the manuscript, insightful suggestions improving the presentation, and last but not least for their support and friendship. Proposition~\ref{propAlmostOpen} originates from an unpublished joint work with Ji\v{r}\'{i} Reif, who was the author's Ph.D. supervisor, initiated his interest in open mapping theorems, and unfortunately passed away before the corresponding thesis was defended.

 \end{document}